\documentclass[amssymb,amsfonts,reqno]{amsart}

 \usepackage{xcolor}
 \usepackage[citebordercolor={green}]{hyperref}
									
\usepackage{color}
\usepackage{latexsym}
\usepackage{amsmath}
\usepackage{amssymb}
\usepackage{amsfonts}
\usepackage{esint}
\usepackage{graphicx}
\usepackage{accents}
\newcommand{\Lie}{{\mathcal{L}}}

\newcommand{\der}{\nabla}

\newcommand{\les}{\lesssim}
\newcommand{\bea}{\begin{eqnarray}}
\newcommand{\eea}{\end{eqnarray}}

\newcommand{\derm}{ { \der^{(\bf{m})}} }
\newcommand{\rderm}{   {\mbox{$\nabla \mkern-13mu /$\,}^{(\bf{m})} }   }

\newcommand{\eps}{{\varepsilon}}\newcommand{\R}{{\mathbb R}}

\newcommand{\E}{{\cal E}}

\newcommand{\la}{\langle}\newcommand{\si}{\sigma}\renewcommand{\b}{\beta}

\newcommand{\cal}{\mathcal}

\def\a{\alpha}\def\ga{\gamma}\def\de{\delta}
\def\bm{\left( \begin{array}{cc}}
\def\endm{\end{array}\right)}\newcommand{\eq}{\end{equation}}
\def\a{\alpha}\def\b{\beta}
\def\ga{\gamma}\def\de{\delta}\def\pa{\partial}

\def \rectangle#1#2{\hbox{\vrule\vbox to #2 {\hrule\hbox to #1{\hfil}\vfil\hrule}\vrule}}
\def\Lb{\underline{L}}
\def\a{\alpha}\def\b{\beta}\def\ga{\gamma}
\def\de{\delta}\def\pa{\partial}
\def\Lb{\underline{L}}

\def\pa{\partial}

\def\beaa{\begin{eqnarray*}}
\def\eeaa{\end{eqnarray*}}
\def\pa{\partial}
\def\a{{\alpha}}
\def\b{{\beta}}
\def\ga{\gamma}

\def\de{\delta}

\def\eps{\epsilon}
\def\la{\lambda}

\def\si{\sigma}

\def\Om{\Omega}

\def\Lb{{\underline{L}}}

\def\g{{\bf g}}

\def\SSS{{\Bbb S}}

\def\R{{\mathbb R}}

\def \p{ \partial}

\def\12{\frac{1}{2}}

\def\bep{\begin{proposition}}
\def\eep{\end{proposition}}

\def\4{\frac{1}{4}}

\def \p{ \partial}

\def\12{\frac{1}{2}}

\def\bep{\begin{proposition}}
\def\eep{\end{proposition}}

\def\bm#1{\boldsymbol{#1}}

\def\build#1_#2^#3{\mathrel{\mathop{\kern 0pt#1}\limits_{#2}^{#3}}}
\def\4{\frac{1}{4}}

\def\<{\langle}
\def\>{\rangle}

\theoremstyle{plain}
\newtheorem{theorem}{Theorem}
\newtheorem{proposition}{Proposition}
\newtheorem{lemma}{Lemma}
\newtheorem{corollary}{Corollary}

\theoremstyle{remark}

\theoremstyle{definition}
\newtheorem{definition}{Definition}
\numberwithin{equation}{section}
\numberwithin{proposition}{section}
\numberwithin{definition}{section}
\numberwithin{lemma}{section}
\numberwithin{corollary}{section}
\numberwithin{remark}{section}
\begin{document}
\include{psfig}
\title[Einstein-Yang-Mills in the Lorenz gauge]{Overview of the proof of the exterior stability of the $\bf{(1+3)}$-Minkowski space-time governed by the Einstein-Yang-Mills system in the Lorenz gauge}
\author{Sari Ghanem}
\address{Beijing Institute of Mathematical Sciences and Applications (BIMSA)}
\email{sarighanem@bimsa.cn}
\maketitle

\begin{abstract}
We study the Einstein-Yang-Mills system in both the Lorenz and harmonic gauges, where the Yang-Mills fields are valued in any arbitrary Lie algebra $\cal G$\,, associated to any compact Lie group $G$\,. This gives a system of hyperbolic partial partial differential that does not satisfy the null condition and that has new complications that are not present for the Einstein vacuum equations nor for the Einstein-Maxwell system. We prove the exterior stability of the Minkowski space-time, $\mathbb{R}^{1+3}$\,, governed by the fully coupled Einstein-Yang-Mills system in the Lorenz gauge, valued in any arbitrary Lie algebra $\cal G$, without any assumption of spherical symmetry. We start with an arbitrary sufficiently small initial data, defined in a suitable energy norm for the perturbations of the Yang-Mills potential and of the Minkowski space-time, and we show the well-posedness of the Cauchy development in the exterior, and we prove that this leads to solutions converging in the Lorenz gauge and in wave coordinates to the zero Yang-Mills fields and to the Minkowski space-time. This provides a first detailed proof of the exterior stability of Minkowski governed by the fully non-linear Einstein-Yang-Mills equations in the Lorenz gauge, by using a null frame decomposition that was first used by H. Lindblad and I. Rodnianski for the case of the Einstein vacuum equations. We note that in contrast to the much simpler case of the Einstein-Maxwell equations where one can omit the potential, in fact in the non-abelian case of the Einstein-Yang-Mills equations, the question of stability, or non-stability, is a purely gauge dependent statement and the partial differential equations depend on the gauge on the Yang-Mills potential that is needed to write up the equations.
\end{abstract}

\setcounter{page}{1}
\pagenumbering{arabic}

\section{Introduction}\label{BrielfSummary}

We are interested to study the structure of the partial differential equations for the fully coupled Einstein-Yang-Mills equations in the Lorenz gauge on the Yang-Mills potential. Indeed, in the Lorenz gauge, one gets a system of non-linear hyperbolic partial differential equations that does not satisfy the null condition, and that has a structure that is different than the Einstein vacuum equations in wave coordinates and different than the Einstein-Maxwell equations in the Lorenz gauge, and one faces serious obstacles for proving dispersive estimates for the Yang-Mills fields.

In fact, unlike the case of the Maxwell equations, in the case of the Yang-Mills fields, gauge transformations change the nature of the partial differential equations on the Yang-Mills potential $A$\,, that is needed to write up the Yang-Mills equations. Whereas in the abelian case of the Maxwell equations, one can omit the potential $ A$ in writing the equations and one can consider only the curvature $ F $ valued in the Lie algebra, yet one cannot do that for the Yang-Mills fields. 

Therefore, in the much more complicated case of the Einstein-Yang-Mills fields, the question of stability, or non-stability, is a purely gauge dependent statement and the partial differential equations depend on the potential $ A $, in contrast to the Einstein-Maxwell equations where one can study the Lie algebra curvature directly. In the Lorenz gauge, the Einstein-Yang-Mills equations do \textit{not} have a null structure and present serious complications that do not exist at all neither for the Einstein vacuum equations nor for the Einstein-Maxwell equations. Our aim in overcoming these complications is to advance on the study of structures for non-linear hyperbolic partial differential equations.

We prove the exterior stability of the $(1+3)$-Minkowski space-time governed by the evolution problem in General Relativity with matter, which in this case is the Einstein-Yang-Mills system in the Lorenz gauge (see \cite{G4}), which reads the following system on the unknown $(\cal M, A, \g)$\;,
\bea
\begin{cases} \label{EYMsystemforintro}
R_{ \mu \nu}  &= 2 < F_{\mu\b}, F_{\nu}^{\;\;\b} >- \frac{1}{2} \cdot g_{\mu\nu } < F_{\a\b },F^{\a\b } >\; ,   \\
0 \;\;\;\, \,&=  \der_{\alpha} F^{\a\b}  + [A_{\alpha}, F^{\a\b} ]    \;, \\
F_{\a\b} &= \der_{\a}A_{\b} - \der_{\b}A_{\a} + [A_{\a},A_{\b}] \; \\
\der^\a A_\a &= 0 \; ,  \end{cases} 
      \eea
where $\cal M$ is the unknown manifold, where $A$ is the unknown Yang-Mills potential valued in the Lie algebra $\cal G$ associated to any compact Lie group $G$\;, where $\g$ is the unknown Lorentzian metric, and where $\der_{\alpha}$ is the unknown space-time covariant derivative of Levi-Civita prescribed by $\g$ and $R_{ \mu \nu} $ is the Ricci tensor. We start with a general initial data for the Yang-Mills potential $A$ chosen to be small, and with general initial data for the metric $\g$ that is asymptotical flat and chosen to be close to the Minkowski initial data. Then, we show that in the Lorenz gauge and in wave coordinates, the perturbations disperse in time in the complement of the future causal domain of a compact set, and lead to a solution that is converging to the Minkowski space-time in the exterior.

\subsection{Definitions for the perturbations of the Minkowski metric}\

\begin{definition}\label{definitionoftheMinkwoskimetricminwavecoordinates}
In wave coordinates,  $\{x^0, x^1, \ldots, x^n \}$ ,we define $m$ to be Minkowski metric. This means that in the system of wave coordinates $\{x^0, ..., x^n\}$\;, we prescribe $m$ by
\beaa
m_{00}&=&-1 \;,\qquad  m_{ii}=1\;,\quad \text{if}\quad i=1, ...,n,\\
\quad\text{and}\quad m_{\mu\nu}&=&0\;,\quad \text{if}\quad \mu\neq \nu \quad \text{for} \quad  \mu, \nu \in \{0, 1, ..., n\} .
\eeaa
We define $h$ as the 2-tensor given by:
\bea
h_{\mu\nu} = g_{\mu\nu} - m_{\mu\nu} \, .
\eea
Let $m^{\mu\nu}$ be the inverse of $m_{\mu\nu}$. We define
\bea\label{definitionofsmallh}
h^{\mu\nu} &=& m^{\mu\mu^\prime}m^{\nu\nu^\prime}h_{\mu^\prime\nu^\prime} \\
\label{definitionsofbigH}
H^{\mu\nu} &=& g^{\mu\nu}-m^{\mu\nu}.
\eea
We define $\der^{(\bf{m})}$ to be the covariant derivative associated to the flat metric $m$\;. Given the definition of $m$ in Definition \ref{definitionoftheMinkwoskimetricminwavecoordinates}, the Christoffel symbols are vanishing in wave coordinates, and therefore, for all $\mu, \nu \in \{0, 1, \ldots, n \}$\;,
 \bea
{ \der^{(\bf{m})}}_{ \frac{\pa}{\pa x^\mu}}  \frac{\pa}{\pa x^\nu} := 0 \; .
 \eea
Let $Z_{\a\b} = x_{\b} \pa_{\a} - x_{\a} \pa_{\b}  \; ,$ $S = t \pa_t + \sum_{i=1}^{3} x^i \pa_{i} \; .$ The Minkowski vector fields will be denoted by $Z$\;, and they are defined such that
\bea
Z \in {\cal Z}  := \big\{ Z_{\a\b}\,,\, S\,,\, \pa_{\a} \, \,  | \, \,   \a\,,\, \b \in \{ 0, \ldots, 3 \} \big\} \; .
\eea
The family ${\cal Z}$ has  $11$ vector fields: $6$ vectors for the Lorentz boosts and rotations, $4$ space-time translations and one scaling vector field. One can order them and assign to each vector an $11$-dimensional integer index $(0, \ldots, 1, \ldots,0)$. Hence, a collection of $k$ vector fields from the family ${\cal Z}$, can be described by the set $I=(\iota_1, \ldots,\iota_k)$, where each $\iota_i$ is an $11$-dimensional integer, where $|I|=k = \sum_{i=1}^{k} | \iota_i |$, with $|\iota_i|=1$\;. We define
\bea
Z^I :=Z^{\iota_1}\ldots Z^{\iota_k} \quad \text{for} \quad I=(\iota_1, \ldots,\iota_k),  
\eea
where $\iota_i$ is an 11-dimensional integer index, with $|\iota_i|=1$, and $Z^{\iota_i}$ representing each a vector field from the family ${\cal Z}$. For a tensor $T$, of arbitrary order, either a scalar or valued in the Lie algebra, we define the Lie derivative as
\bea
\Lie_{Z^I} T :=\Lie_{Z^{\iota_1}} \ldots \Lie_{Z^{\iota_k}} T \quad \text{for} \quad I=(\iota_1, \ldots,\iota_k) .
\eea
By $\sum_{|I| \leq k}$ we mean the sum over all possible products of length at most $k$, of Minkowski vector fields. We define $E_{\mu\nu}$ as the Euclidian metric in wave coordinates. We then define for a  tensor of arbitrary order, for example $K_{\a}$\;, 
\bea
 | \pa K |^2 := | \derm K |^2 :=  E^{\a\b} E^{\mu\nu}    { \der^{(\bf{m})}}_{\mu} K_\a \ \cdot    { \der^{(\bf{m})}}_{\nu} K_\b \; .
 \eea
We then have
 \bea
 \notag
| \pa K |^2  &=&   |  {\der^{(\bf{m})}}_{t} K |^2  +  |  {\der^{(\bf{m})}}_{x^1} K |^2 +\ldots  + |  {\der^{(\bf{m})}}_{x^n} K |^2 =  \sum_{\a,\; \b \in  \{t, x^1, \ldots, x^n \}} |  \pa_{\a} K_\b |^2   \, ,
\eea
\end{definition}
In this paper, we will prove the theorem.

\subsection{The theorem}\label{Thetheoremofexteriorstabilityfornequalthree}\

\begin{theorem}
Assume that we are given an initial data set $(\Sigma, \overline{A}, \overline{E}, \overline{g}, \overline{k})$ for \eqref{EYMsystemforintro}. We assume that $\Sigma$ is diffeomorphic to $\R^3$\;. Then, there exists a global system of coordinates $(x^1, x^2, x^3) \in \R^3$ for $\Sigma$\;. We define
\bea
r := \sqrt{ (x^1)^2 + (x^2)^2 +(x^3)^2  }\;.
\eea
Furthermore, we assume that the data $(\overline{A}, \overline{E}, \overline{g}, \overline{k}) $ is smooth and asymptotically flat. Let $\chi$ be a smooth function such as
 \bea\label{defXicutofffunction}
\chi (r)  := \begin{cases} 1  \quad\text{for }\quad r \geq \frac{3}{4} \;  ,\\
0 \quad\text{for }\quad r \leq \frac{1}{2} \;. \end{cases} 
\eea
Let $M$ be the mass such as $ 0 < M \leq \eps^2 \leq 1 $\;, let $\de_{ij}$ be the Kronecker symbol, and let $\overline{h}^1_{ij} $ be defined in this system of coordinates $x^i$\;, by
 \bea
\overline{h}^1_{ij} := \overline{g}_{ij} - (1 + \chi (r)\cdot  \frac{ M}{r}  ) \de_{ij} \; .
\eea
We then define the weighted $L^2$ norm on $\Sigma$\;, namely $\overline{\E}_N$\;, for $\ga > 0$\;, by
 \bea\label{definitionoftheenergynormforinitialdata}
 \notag
\overline{\E}_N &:=&  \sum_{|I|\leq N} \big(   \| (1+r)^{1/2 + \ga + |I|}   \overline{D} (  \overline{D}^I  \overline{A}    )  \|_{L^2 (\Sigma)} +  \|(1+r)^{1/2 + \ga + |I|}    \overline{D}  ( \overline{D}^I \overline{h}^1   )  \|_{L^2 (\Sigma)} \big) \\
\notag
&:=&  \sum_{|I|\leq N}      \big(   \sum_{i=1}^{n}  \|(1+r)^{1/2 + \ga + |I|}     \overline{D} (  \overline{D}^I  \overline{A_i}    )  \|_{L^2 (\Sigma)} +  \sum_{i, j =1}^{n}  \|(1+r)^{1/2 + \ga + |I|}     \overline{D}  ( \overline{D}^I \overline{h}^1_{ij}   )  \|_{L^2 (\Sigma)} \big) \; , \\
\eea
where the integration is taken on $\Sigma$ with respect to the Lebesgue measure $dx_1 \ldots dx_n$\;, and where $\overline{D} $ is the Levi-Civita covariant derivative associated to the given Riemannian metric $\overline{g}$\;. We also assume that the initial data set $(\Sigma, \overline{A}, \overline{E}, \overline{g}, \overline{k})$ satisfies the Einstein-Yang-Mills constraint equations, namely
\bea\label{constraint1}
  \mathcal{R}+ \overline{k}^i_{\,\, \, i} \overline{k}_{j}^{\,\,\,j}  -  \overline{k}^{ij} \overline{k}_{ij}   &=&    \frac{4}{(n-1)}   < \overline{E}_{i}, \overline{E}^{ i}>   \\
 \notag
 && +  < \overline{D}_{i}  \overline{A}_{j} - \overline{D}_{j} \overline{A}_{i} + [ \overline{A}_{i},  \overline{A}_{j}] ,\overline{D}^{i}  \overline{A}^{j} - \overline{D}^{j} \overline{A}^{i} + [ \overline{A}^{i},  \overline{A}^{j}] >  \;  ,\\
 \label{constraint2}
\overline{D}_{i} \overline{k}^i_{\,\,\, j}    - \overline{D}_{j} \overline{k}^i_{\, \,\,i}  &=&  2 < \overline{E}_{i}, \overline{D}_{j}  \overline{A}^{i} - \overline{D}^{i} \overline{A}_{j} + [ \overline{A}_{j},  \overline{A}^{i}]  >  \;  ,\\
\label{constraint3}
\overline{D}^i \overline{E}_{ i} + [\overline{A}^i, \overline{E}_{ i} ]  &=& 0  \;  .
\eea
For any $N \geq 11$\;, there exists a constant $ \overline{c} (\cal{K}, N, \ga) $\;, that depends on $\cal{K}$\;, on $\ga$\;, and on $N$\;, such that if
\bea\label{Assumptiononinitialdataforglobalexistenceanddecay}
\overline{\E}_{N+2} &\leq&   \overline{c} (\cal{K}, N, \ga)  \; ,\\
M  &\leq&   \overline{c} (\cal{K}, N, \ga)  \; ,
\eea
then there exists a solution $(\cal{M}, A, g)$ to the Cauchy problem for the fully coupled Einstein-Yang-Mills system \ref{EYMsystemforintro} in the future of the whole causal complement of any compact $\cal{K} \subset \Sigma$ \;, converging to the null Yang-Mills potential and to the Minkowski space-time in the following sense: if we define the metric $ m_{\mu\nu}$ to be the Minkowski metric in wave coordinates $(x^0, x^1, x^3)$\, and define $t = x^0$\;, and if we define in this system of wave coordinates 
\bea
h^{1}_{\mu\nu} := g_{\mu\nu} - m_{\mu\nu}- h^0_{\mu\nu}  \;  ,
\eea
where for $t > 0$, 
\bea\label{guessonpropagationofthesphericallsymmetricpart}
h^0_{\mu\nu} := \chi(r/t) \cdot \chi(r)\cdot \frac{M}{r}\de_{\mu\nu}  \;  ,
\eea
and where for $t=0$\;,  \bea\label{definitionofthesphericallysymmtericpartofinitialdata}
h^0_{\mu\nu} ( t= 0) := \chi(r)\cdot \frac{M}{r}\cdot \de_{\mu\nu}  \; ,
\eea
then, for $\overline{h}^1_{ij} $ and $\overline{A}_i$ decaying sufficiently fast, we have the following estimates on $h^1$\;, and on $A$ in the Lorenz gauge, for the norm constructed using wave coordinates by taking the sum over all indices in wave coordinates, given in \eqref{boundonNLiederivativeofgradient}, \eqref{boundonNLiederivativeofzerothderivative}, \eqref{boundonNLiederivativeotheYangMillscurvature}. That there exists a constant $C(N)$ to bound  $N-2$ Lie derivative of the fields in direction of Minkowski vector fields, and to bound the growth of $\E _{N} (\cal K) (t)$ in \eqref{theboundinthetheoremonEnbyconstantEN}, and that there exists a constant $\eps$ that depends on $\overline{c} (\cal{K}, N, \ga)$\;, on  $\cal{K} \subset \Sigma$\;, on $N$ and on $\ga$\;, such that we have the following estimates in the whole complement of the future causal of the compact $\cal{K}  \subset \Sigma$\;, for all $|I| \leq N -2$\;, 
\bea\label{boundonNLiederivativeofgradient}
 \notag
  && \sum_{\mu= 0}^{n} |\derm  ( \Lie_{Z^I}  A_{\mu} ) (t,x)  |     +   \sum_{\mu, \nu = 0}^{n}  |\derm  ( \Lie_{Z^I}  h^1_{\mu\nu} ) (t,x)  |   \\
   \notag
    &\leq&  C(\cal{K} ) \cdot   C(N)  \cdot \frac{\eps }{(1+t+|r-t|)^{1-\eps} (1+|r-t|)^{1+\gamma}} \; ,\\
      \eea
and
 \bea\label{boundonNLiederivativeofzerothderivative}
 \notag
  \sum_{\mu= 0}^{n}  |\Lie_{Z^I} A_{\mu} (t,x)  | +   \sum_{\mu, \nu = 0}^{n}  |\Lie_{Z^I}  h^1_{\mu\nu} (t,x)  |  &\leq&C(\cal{K} ) \cdot C(N) \cdot c(\ga) \cdot \frac{\eps }{(1+t+|r-t|)^{1-\eps} (1+|r-t|)^{\gamma}} \; , \\
      \eea
      where $Z^I$ are the Minkowski vector fields. In particular, the gauge invariant norm on the Yang-Mills curvature decays as follows, for all $|I| \leq N - 2$\;,
 \bea\label{boundonNLiederivativeotheYangMillscurvature}
 \notag
 \sum_{\mu, \nu = 0}^{n}  |\Lie_{Z^I} F_{\mu\nu}  (t,x) |  &\leq&C(\cal{K} ) \cdot   C(N)    \cdot \frac{\eps }{(1+t+|r-t|)^{1-\eps} (1+|r-t|)^{1+\gamma}}  \\
  \notag
&& +     C(\cal{K} ) \cdot C(N) \cdot c(\ga)   \cdot \frac{\eps }{(1+t+|r-t|)^{2-2\eps} (1+|r-t|)^{2\gamma}} \; . \\
      \eea
      Furthermore, if one defines $w$ as follows, 
\bea\label{defoftheweightw}
w(r-t):=\begin{cases} (1+|r-t|)^{1+2\gamma} \quad\text{when }\quad r-t>0 \;, \\
         1 \,\quad\text{when }\quad r-t<0 \; , \end{cases}
\eea
and if we define  $\Sigma_t^{ext} (\cal K)$ as being the time evolution in wave coordinates of $\Sigma$ in the future of the causal complement of $\cal{K}$\;, then for all time $t$\;, we have
\bea\label{theboundinthetheoremonEnbyconstantEN}
\notag
\E _{N} (\cal K) (t) &:=&  \sum_{|J|\leq N} \big( \|w^{1/2}   \derm ( \Lie_{Z^J} h^1   (t,\cdot) )  \|_{L^2 (\Sigma_t^{ext} (\cal K)) } +  \|w^{1/2}   \derm ( \Lie_{Z^J}  A   (t,\cdot) )  \|_{L^2 (\Sigma_t^{ext} (\cal K))  }  \big) \\
&\leq& C(N) \cdot \eps \cdot (1+t)^{\eps} \; . 
\eea
More precisely, for any constant $E(N)$ and for any $\ga > 0$\;, there exists a constant $\de (\ga) > 0$ that depends on $\ga$\;, and there exists a constant $\overline{c}_1 (\cal{K}, N, \ga, \de)$\;, that depends on  $\cal{K} \subset \Sigma$\;, on $E(N)$\;, on $N$\;, on $\ga$ and on $\de$\;, such that if 
\bea\label{Assumptiononinitialdataforglobalexistenceanddecay}
 \overline{\E}_{N+2} (0) &\leq&   \overline{c}_1 (\cal{K}, E (N), N, \ga, \de)   \; ,\\
 M &\leq&  \overline{c}_1 (\cal{K}, E (N), N, \ga, \de)  \; ,
 \eea
then, we have in the whole complement of the future causal of the compact $\cal{K}  \subset \Sigma$\;, for all time $t$\,, 
\bea\label{Theboundontheglobaleenergybyaconstantthatwechoose}
 \E_{N}(\cal K)  (t) \leq E(N) \cdot (1+t)^{\de} \; .
\eea
\end{theorem}

\section{The proof}

The goal of this section is to prove Theorem \ref{Thetheoremofexteriorstabilityfornequalthree}. The proof is decomposed into four parts that address difficulties which are not present for the Einstein vacuum equations nor for the Einstein-Maxwell equations. These parts consist in recasting the Einstein-Yang-Mills system in the Lorenz gauge and in wave coordinates as a coupled system of covariant non-linear wave equations decomposed in a null-frame, a system that have terms that are new to the Einstein-Maxwell equations (Subsection \ref{recastingasystemofnonlinearwaves}), then dealing with the new term $A_{e_a}  \cdot     \derm A_{e_a} $ (Subsection \ref{dealingiwthbadtermAaderAa}), as well as dealing with the new term $ A_L   \cdot   \derm A $ (Subsection \ref{dealingwithALderA}), and upgrading the dispersive estimates for the Lie derivatives with a new separated estimate on the commutator term (Subsection \ref{dealingwithLiederivativesupgrade}).

We define the weight $w$ as in \eqref{defoftheweightw}. We define the higher order energy norm as the following $L^2$ norms on $A$ and $h^1$ in the exterior, 
 \bea\label{definitionoftheenergynorm}
 \notag
\E_{N} (t) :=  \sum_{|I|\leq N} \big(   \|w^{1/2}   \derm ( \Lie_{Z^I}  A   (t,\cdot) )  \|_{L^2 (\Sigma_t^{ext}) } +  \|w^{1/2}   \derm ( \Lie_{Z^I} h^1   (t,\cdot) )  \|_{L^2 (\Sigma_t^{ext})} \big) \, , \\
\eea
where the integration is taken with respect to the Lebesgue measure $dx_1 \ldots dx_n$\:. We start with an \`a priori estimate (see \eqref{aprioriestimate}) and we want to upgrade this estimate to ultimately prove that it is a true estimate. We look at any time $T \in [0, T_{\text{loc}})$, such that for all $t$ in the interval of time $[0, T]$, we have
\bea\label{aprioriestimate}
\E_{ N } (t)  \leq E (N )  \cdot \eps \cdot (1 +t)^\delta \;,
\eea
where $E ( N )$\, is a constant that depends on $ N $\,, where $\eps > 0$ is a constant to be chosen later small enough, and where $\delta \geq 0$ is to be chosen later. 

\begin{definition}\label{definitionofbigOforLiederivatives}
For a family of tensors Let $\Lie_{Z^{I_1}}K^{(1)}, \ldots, \Lie_{Z^{I_m}} K^{(m)}$, where each tensor $ K^{(l)}$ is again either $A$ or $h$ or $H$\;, or $\derm A$\;, $\derm h$ or $\derm H$\;, we define
\bea
\notag
&& O_{\mu_{1} \ldots \mu_{k} } (\Lie_{Z^{I_1}}K^{(1)} \cdot \ldots \cdot \Lie_{Z^{I_m}} K^{(m)}  ) \\
\notag
&:=& \prod_{l=1}^{m} \Big[  \prod_{|J_l| \leq |I_l|} Q_{1}^{J_l} ( \Lie_{Z^{J_l}} K^{(l)} ) \cdot \Big( \sum_{n=0}^{\infty} P_{n}^{J_l}  ( \Lie_{Z^{J_l}} K^{(l)} ) \Big) \Big] \; . \\
\eea
where again $P_n^{J_l} (K^{l} )$ and $Q_1^{J_l} (K)$, are tensors that are Polynomials of degree $n$ and $1$, respectively, with $Q^{J_l}_1 (0) = 0$ and $Q^{J_l}_1 \neq 0$\;, of which the coefficients are components in wave coordinates of the metric $\textbf m$ and of the inverse metric $\textbf m^{-1}$, and of which the variables are components in wave coordinates of the covariant tensor $\Lie_{Z^{J_l}} K^{l}$, leaving some indices free, so that at the end the whole product $$  \prod_{l=1}^{m} \Big[  \prod_{|J_l| \leq |I_l|}  Q_{1}^{J_l} ( \Lie_{Z^{J_l}} K^{(l)} ) \cdot \Big( \sum_{n=0}^{\infty} P_{n}^{J_l}  ( \Lie_{Z^{J_l}} K^{(l)} ) \Big) \Big]$$ gives a tensor with free indices $\mu_{1} \ldots \mu_{k}$\;. To lighten the notation, we shall drop the indices and just write $O (\Lie_{Z^{I_1}}K^{(1)} \cdot \ldots \cdot \Lie_{Z^{I_m}} K^{(m)}  )$.
\end{definition}
We recall the following lemma from a previous paper, see \cite{G4}.
\begin{lemma}\label{linkbetweenbigHandsamllh}
We have
\bea
H^{\mu\nu}= -h^{\mu\nu}+O^{\mu\nu}(h^2) \, .
\eea
\end{lemma}

On one hand, the Einstein-Yang-Mills system in Lorenz gauge imply that 
      \bea\label{ThewaveequationontheYangMillspotentialwithhyperbolicwaveoperatorusingingpartialderivativesinwavecoordinates}
  \notag
 && g^{\la\mu} \derm_{\la}   \derm_{\mu}   A_{\si}     \\
   \notag
 &=&    ( \derm_{\si}  h^{\a\mu} )  \cdot (  \derm_{\a}A_{\mu} )    \\
 \notag
&&  +   \frac{1}{2}    \big(   \derm^{\mu} h^{\nu}_{\, \, \, \si} +  \derm_\si h^{\nu\mu} -   \derm^{\nu} h^{\mu}_{\,\,\, \si}  \big)   \cdot  \big( \derm_{\mu}A_{\nu} -  \derm_{\nu}A_{\mu}  \big) \\
 \notag
&&     +   \frac{1}{2}    \big(   \derm^{\mu} h^{\nu}_{\, \, \, \si} +  \derm_\si h^{\nu\mu} -   \derm^{\nu} h^{\mu}_{\,\,\, \si}  \big)   \cdot   [A_{\mu},A_{\nu}] \\
 \notag
 && -  \big(  [ A_{\mu}, \derm^{\mu} A_{\si} ]  +    [A^{\mu},  \derm_{\mu}  A_{\si} - \derm_{\si} A_{\mu} ]    +    [A^{\mu}, [A_{\mu},A_{\si}] ]  \big)  \\
 \notag
  && + O( h \cdot  \derm h \cdot  \derm A) + O( h \cdot  \derm h \cdot  A^2) + O( h \cdot  A \cdot \derm A) + O( h \cdot  A^3) \, . \\
  \eea
  and 
  \bea\label{Thewaveequationonthemetrichwithhyperbolicwaveoperatorusingingpartialderivativesinwavecoordinates}
\notag
 && g^{\alpha\beta}\derm_\alpha \derm_\beta h_{\mu\nu} \\
 \notag
  &=& P(\derm_\mu h,\derm_\nu h)  +  Q_{\mu\nu}(\derm h,\derm h)   + O_{\mu\nu}  (h \cdot (\derm h)^2)  \\
\notag
 &&   -4     <   \derm_{\mu}A_{\b} - \derm_{\b}A_{\mu}  ,  \derm_{\nu}A^{\b} - \derm^{\b}A_{\nu}  >    \\
 \notag
 &&   + m_{\mu\nu }       \cdot  <  \derm_{\a}A_{\b} - \derm_{\b}A_{\a} , \derm_{\a} A^{\b} - \derm^{\b}A^{\a} >   \\
 \notag
&&           -4  \cdot  \big( <   \derm_{\mu}A_{\b} - \derm_{\b}A_{\mu}  ,  [A_{\nu},A^{\b}] >   + <   [A_{\mu},A_{\b}] ,  \derm_{\nu}A^{\b} - \derm^{\b}A_{\nu}  > \big)  \\
\notag
&& + m_{\mu\nu }    \cdot \big(  <  \derm_{\a}A_{\b} - \derm_{\b}A_{\a} , [A^{\a},A^{\b}] >    +  <  [A_{\a},A_{\b}] , \derm^{\a}A^{\b} - \derm^{\b}A^{\a}  > \big) \\
\notag
 &&  -4     <   [A_{\mu},A_{\b}] ,  [A_{\nu},A^{\b}] >      + m_{\mu\nu }   \cdot   <  [A_{\a},A_{\b}] , [A^{\a},A^{\b}] >  \\
     && + O \big(h \cdot  (\derm A)^2 \big)   + O \big(  h  \cdot  A^2 \cdot \derm A \big)     + O \big(  h   \cdot  A^4 \big)  \,  , 
\eea
where $P$ and $Q$ are defined as follows:

\beaa\label{definitionofthetermbigPinsourcetermsforeinstein}
 P(\derm_\mu h,\derm_\nu h) :=\frac{1}{4} m^{\alpha\alpha^\prime}\derm_\mu h_{\alpha\alpha^\prime} \, m^{\beta\beta^\prime}\derm_\nu h_{\beta\beta^\prime}  -\frac{1}{2} m^{\alpha\alpha^\prime}m^{\beta\beta^\prime} \derm_\mu h_{\alpha\beta}\, \derm_\nu h_{\alpha^\prime\beta^\prime} \, , 
\eeaa
\beaa\label{definitionofthetermbigQinsourcetermsforeinstein}
\notag
&& Q_{\mu\nu}(\derm h,\derm h) \\
\notag
&:=& \derm_{\alpha} h_{\beta\mu}\, \, m^{\alpha\alpha^\prime}m^{\beta\beta^\prime} \derm_{\alpha^\prime} h_{\beta^\prime\nu} -m^{\alpha\alpha^\prime}m^{\beta\beta^\prime} \big(\pa_{\alpha} h_{\beta\mu}\,\,\pa_{\beta^\prime} h_{\alpha^\prime \nu} -\derm_{\beta^\prime} h_{\beta\mu}\,\,\derm_{\alpha} h_{\alpha^\prime\nu}\big)\\
\notag
&& +m^{\a\a'}m^{\b\b'}\big (\derm_\mu h_{\a'\b'}\, \derm_\a h_{\b\nu}- \derm_\a h_{\a'\b'} \,\derm_\mu h_{\b\nu}\big )  \\
\notag
&& + m^{\a\a'}m^{\b\b'}\big (\derm_\nu h_{\a'\b'} \,\derm_\a h_{\b\mu} - \derm_\a h_{\a'\b'}\, \derm_\nu h_{\b\mu}\big )\\
\notag
&& +\frac 12 m^{\a\a'}m^{\b\b'}\big (\derm_{\b'} h_{\a\a'}\, \derm_\mu h_{\b\nu} - \derm_{\mu} h_{\a\a'}\, \derm_{\b'} h_{\b\nu} \big ) \\
&& +\frac 12 m^{\a\a'}m^{\b\b'} \big (\derm_{\b'} h_{\a\a'}\, \derm_\nu h_{\b\mu} - \derm_{\nu} h_{\a\a'} \,\derm_{\b'} h_{\b\mu} \big ) \, . 
\eeaa

On the other hand, assume that we are given an initial data set $(\Sigma, \overline{A}, \overline{E}, \overline{g}, \overline{k})$ that satisfies the Einstein-Yang-Mills constraint equations given in \eqref{constraint1}, \eqref{constraint3},  \eqref{constraint3}, then, we can construct a new hyperbolic initial data set $(\Sigma, A_\Sigma, \pa_t A_\Sigma, g_\Sigma, \pa_t g_\Sigma)$, as prescribed as follows in
\ref{initialdatadforzerothderivativeAsigma}, \ref{constructionopatialtimeAonsigma}, \ref{constructionoggsigma} and \ref{constructionopatialtgonsigma}, for the coupled system of non-linear hyperbolic wave equations given in \eqref{ThewaveequationontheYangMillspotentialwithhyperbolicwaveoperatorusingingpartialderivativesinwavecoordinates} and \eqref{Thewaveequationonthemetrichwithhyperbolicwaveoperatorusingingpartialderivativesinwavecoordinates}, such that solutions to the Einstein-Yang-Mills equations that solve the system in \eqref{ThewaveequationontheYangMillspotentialwithhyperbolicwaveoperatorusingingpartialderivativesinwavecoordinates} and \eqref{Thewaveequationonthemetrichwithhyperbolicwaveoperatorusingingpartialderivativesinwavecoordinates} are indeed in the Lorenz gauge and in wave coordinates for all time $t$\;. 
 \bea\label{initialdatadforzerothderivativeAsigma}
A_{\Sigma} &=& \begin{cases} (A_{\Sigma})_{t} = 0 \; ,  \\
(A_{\Sigma})_{i}= \overline{A}_{i} \quad\text{prescribed arbitrarily for }\quad i \neq t \; , \end{cases} \\
\label{constructionoggsigma}
g_{\Sigma} &=& \begin{cases} (g_{\Sigma})_{tt} = - N^2 \, , \\
(g_{\Sigma})_{ij} =   \overline{g}_{ij}  \quad\text{given by the initial data}, \\
(g_{\Sigma})_{tj} =  (g_{\Sigma})_{jt} = 0\, , \end{cases} \\
\label{constructionopatialtimeAonsigma}
\p_{t} A_{\Sigma} &=& \begin{cases} (\pa_{t } A_{\Sigma})_{t} =  N^2 \overline{g}^{ij}  \pa_{i}    \overline{A}_{j}    \, ,\\
(\pa_{t } A_{\Sigma})_{i} =  N \overline{E}_{i} \;\text{ where $\overline{E}_i$ is prescribed arbitrarily for } i \neq t \,,\\
\quad\quad\quad\quad\quad\quad\;\;\quad  \text{such that } D^i \overline{E}_{ i} + [\overline{A}^i, \overline{E}_{ i} ] = 0 \,, \\ \end{cases} \\
\label{constructionopatialtgonsigma}
\pa_{t} {g_\Sigma} &=& \begin{cases}  ( \pa_{t} {g_\Sigma})_{ij}  =  2 N \overline{k}_{ij}  \, ,\\
 ( \pa_{t} {g_\Sigma})_{tt} =  - N^2  g^{ij}  \pa_t g_{ij}   =   - 2 N^3  \overline{g}^{ij} \overline{k}_{ij} \, ,\\
  ( \pa_{t} {g_\Sigma})_{tj} =   ( \pa_{t} {g_\Sigma})_{jt} =  - N \pa_j N   + N^2  \overline{g}^{k i}  \pa_k  \overline{g}_{j i} -  \frac{N^2}{2}   \overline{g}^{ik}  \pa_j  \overline{g}_{ik}    \, ,
 \end{cases} 
\eea
where $N$ is an arbitrary lapse function on the Cauchy hypersurface $\Sigma$\,.

\subsection{The system of non-linear hyperbolic partial differential equations in a null-frame}\label{recastingasystemofnonlinearwaves}\

We start by defining a null-frame tetrad $\{\, \underline{L}\,, L\,, e_a\,, a \in \{1, 2 \}\, \}$\;, that is a frame constructed using wave coordinates (see the following Definition \ref{definitionofthenullframusingwavecoordinates}).
\begin{definition}\label{definitionofthenullframusingwavecoordinates}
At a point $p$ in the space-time, let
\bea
L &=&  \pa_{t} + \pa_{r} \, =   \pa_{t} +  \frac{x^{i}}{r} \pa_{i} \, , \\
\underline{L} &=&  \pa_{t} - \pa_{r} =  \pa_{t} -  \frac{x^{i}}{r} \pa_{i} \, ,
\eea
and let $\{e_1, e_{2} \}$ be an orthonormal frame on $\SSS^{2}$. 
We define the sets 
\bea
\cal T&=& \{ L,e_1, e_{2}\} \, ,\\
 \cal U&=&\{ \underline{L}, L, e_1, e_{2}\} \, .
\eea
\end{definition}
Using the Lorenz gauge condition, that is a condition that breaks the gauge invariance of the Einstein-Yang-Mills equations, we get the following estimates. That is that in the Lorenz gauge, the potential $A$ satisfies the following inequalities
\bea\label{estimateoncovariantgradientofAL}
| \derm  A_{L}  | &\les&  | \rderm A  | \, + O \big(  | h  | \cdot  | \derm  A  | \big) \; ,
\eea
and
\bea\label{LorenzgaugeestimateforgradientofLiederivativesofAL}
\notag
 | \derm ( \Lie_{Z^J}  A_{L}  ) | &\les& \sum_{|I| \leq |J|} | \rderm ( \Lie_{Z^I} A ) | \, + \sum_{|K| + |M| \leq |J|} O \big(  | \Lie_{Z^K}  h   | \cdot  | \derm  (  \Lie_{Z^M}  A ) | \big)  \;, 
\eea
where
\bea
 | \derm ( \Lie_{Z^J}  A_\mu )   |^2  :=  E^{\a\b} < \derm_{\b}( \Lie_{Z^J}   A_{\mu} ) \, ,  \derm_{\a} ( \Lie_{Z^J}  A_{\mu} )  > \; ,
\eea
and where $L$ and $\rderm$ are constructed using wave coordinates, and here  $\rderm$ is the covariant derivative restricted on the 2-spheres $\SSS^{2}$. In fact, estimates \eqref{estimateoncovariantgradientofAL} and \eqref{LorenzgaugeestimateforgradientofLiederivativesofAL}), are not only estimates, but one can express literally the terms on the left hand side of the inequality as a combination of terms in the right hand side of the inequality). Using the wave coordinates, that is a system of coordinates for the Einstein-Yang-Mills equations satisfying some condition, namely the harmonic gauge condition, that breaks the diffeomorphism invariance of the Einstein-Yang-Mills equations, we get the following estimates \eqref{waveconditionestimateonzeroLiederivativeofmetric} and \eqref{wavecoordinatesestimateonLiederivativesZonmetric}. That is, we have
\bea\label{waveconditionestimateonzeroLiederivativeofmetric}
| \derm  H_{\cal T L} | &\les& | \rderm  H |  + O (|H| \cdot |\derm H| ) \; ,
\eea
and
\beaa
| \derm  h_{\cal T L} | &\les& | \rderm  h |  + O (|h| \cdot |\derm h| ) \; .
\eeaa
Also,
\bea\label{wavecoordinatesestimateonLiederivativesZonmetric}
\notag
| \derm ( \Lie_{Z^J}  H_{\cal T L} )  | &\les&\sum_{|K| \leq |J| }  |  \rderm ( \Lie_{Z^K} H ) |  + \sum_{|K|+ |M| \leq |J|}  O (| \Lie_{Z^K} H| \cdot |\derm ( \Lie_{Z^M} H ) | ) \; , \\
\eea
and
\beaa
| \derm ( \Lie_{Z^J} h_{\cal T L} )  | &\les& \sum_{|K| \leq |J| } | \rderm  ( \Lie_{Z^K}  h)  |  + \sum_{|K|+ |M| \leq |J|}  O (|\Lie_{Z^K} h| \cdot |\derm ( \Lie_{Z^M} h ) | )\; .
\eeaa
Estimates \eqref{waveconditionestimateonzeroLiederivativeofmetric} and \eqref{wavecoordinatesestimateonLiederivativesZonmetric}, are not only estimates but they are actually equalities between tensors, therefore one can write the Einstein-Yang-Mills equations as a coupled system of non-linear hyperbolic equations on the Einstein-Yang-Mills potential $A$\;, that is a one-tensor valued in the Lie algebra ${\cal G}$\;, and on the perturbation metric $h^1 = g - m - h^0$\;, that is a scalar valued two-tensor, where $g$ is the unknown metric solution to the evolution problem for the Einstein-Yang-Mills system, and where $h^0$ is the spherically symmetric Schwarzschildian part of the perturbation defined in \eqref{guessonpropagationofthesphericallsymmetricpart}.

To study the structure of the non-linearities in the source terms of these hyperbolic operators $g^{\la\mu} \pa_{\la}   \pa_{\mu}   A_{\si} $ and $ g^{\alpha\beta}\pa_\alpha\pa_\beta
g_{\mu\nu}$, we use the null frame that will help us decompose the terms into “good” terms and “bad” terms: the good terms will be terms that we could control using estimates that exploit the Lorenz gauge and the wave coordinates conditions -- this will be apparent by the control by “good derivatives” on the “good terms” that we will show. We recall that $\cal T:= \{\, L\,, e_a\,, a \in \{1, 2 \}\, \}$ is a set of three vectors of the frame tangent to the outgoing null-cone for the Minkowski metric $m$\;, where $m$ is defined to be Minkowski metric in wave coordinates, and that $\cal U := \{\, \underline{L}\,, {\cal T}\,  \}$ are the full components of the frame. We get a coupled system of non-linear wave equations on the Einstein-Yang-Mills potential $A$ and on the perturbation metric $h^1$, that is schematically speaking the following:
\bea\label{waveequationonbadtermfortheEinsteinYangMillspoential}
   \notag
 &&   g^{\la\mu} \derm_{\la}   \derm_{\mu}   A_{{\underline{L}}}     \\
 \notag
    &=& \derm h  \cdot  \rderm A        + \rderm  h \cdot  \derm A  +    A   \cdot     \rderm A   +  \derm  h  \cdot  A^2   +  A^3 \\
    \notag
  && + O( h \cdot  \derm h \cdot  \derm A) + O( h \cdot  A \cdot \derm A)  + O( h \cdot  \derm h \cdot  A^2) + O( h \cdot  A^3) \\
         \notag
 && +  A_L   \cdot   \derm A      +  A_{e_a}  \cdot     \derm A_{e_a}     \; ,\\
  \eea
 and
    \bea
\notag
&&    g^{\alpha\beta} \derm_\alpha \derm_\beta h^1_{{\underline{L}} {\underline{L}} }     \\
\notag
            &=&   \rderm h \cdot \derm h +  h \cdot ( \derm h )^2      +   \rderm A    \cdot  \derm  A   +  A^2  \cdot  \derm A     +  A^4 \\
    \notag
     && + O \big(h \cdot  (\derm A)^2 \big)   + O \big(  h  \cdot  A^2 \cdot \derm A \big)     + O \big(  h   \cdot  A^4 \big)   \\
         \notag
 && + (  \derm h_{\cal TU} ) ^2 +   ( \derm A_{e_{a}} )^2    \\
 \notag
 && +     g^{\alpha\beta} \derm_\alpha \derm_\beta h^0 \; . 
\eea
and for the “good” components of $A$ and $h^1$, we have the “better” structure that is
  \bea
   \notag
 &&   g^{\la\mu} \derm_{\la}   \derm_{\mu}   A_{{\cal T}}     \\
 \notag
    &=&   \derm h  \cdot  \rderm A      +  \rderm  h  \cdot  \derm A   +    A   \cdot   \rderm A  +  \derm  h  \cdot  A^2   +  A^3 \\
    \notag
  && + O( h \cdot  \derm h \cdot  \derm A) + O( h \cdot  A \cdot \derm A)  + O( h \cdot  \derm h \cdot  A^2) + O( h \cdot  A^3)  \; ,
  \eea
  and
 \bea
\notag
 &&   g^{\alpha\beta} \derm_\alpha \derm_\beta h^1_{ {\cal T} {\cal U}}    \\
 \notag
    &=&  \rderm h \cdot \derm h + h \cdot (\derm h)^2      +  \rderm A   \cdot  \derm  A   +    A^2 \cdot \derm A     + A^4  \\
\notag
     && + O \big(h \cdot  (\derm A)^2 \big)   + O \big(  h  \cdot  A^2 \cdot \derm A \big)     + O \big(  h   \cdot  A^4 \big)   \\
  \notag
  && +    g^{\alpha\beta} \derm_\alpha \derm_\beta h^0  \; .
\eea
We see in the troublesome wave equation in \eqref{waveequationonbadtermfortheEinsteinYangMillspoential} on $A_{{\underline{L}}}$\;, that there exist “bad” terms which are $ A_{e_a}  \cdot     \derm A_{e_a} $\;, where $a \in \{1, 2 \}$\,, and $ A_L   \cdot   \derm A    $\;.

\subsection{Dealing with the “bad” term $A_{e_a}  \cdot     \derm A_{e_a} $\;}\label{dealingiwthbadtermAaderAa}\

To control correctly the term $ |A_{e_a} | \cdot |   \derm A_{e_a} |$\;, we use first an energy estimate that involves only the component $A_{e_a}$\;, an energy estimate that we established in \cite{G5} (see the following Lemma \ref{Theveryfinal } from \cite{G5}).
\begin{lemma}\label{Theveryfinal }
For  $\eps$ small enough, depending on $q_0$\,, on $\de$\,, on $\ga$\;, and on $\mu < 0$\;, the following energy estimate holds for $\ga > 0$ and for $\Phi$ decaying sufficiently fast at spatial infinity,
  \beaa
   \notag
 &&     \int_{\Sigma^{ext}_{t_2} }  |\derm \Phi_{V} |^2     \cdot w(q)  \cdot d^{3}x    + \int_{N_{t_1}^{t_2} }  T_{\hat{L} t}^{(\bf{g})} (\Phi_{V})  \cdot  w(q) \cdot dv^{(\bf{m})}_N \\
 \notag
 &&+ \int_{t_1}^{t_2}  \int_{\Sigma^{ext}_{\tau} }     | \rderm \Phi_{V} |^2   \cdot  \frac{\widehat{w} (q)}{(1+|q|)} \cdot  d^{3}x  \cdot d\tau \\
  \notag
  &\les &       \int_{\Sigma^{ext}_{t_1} }  |\derm \Phi_{V} |^2     \cdot w(q)  \cdot d^{3}x \\
    \notag
   &&  +  \int_{t_1}^{t_2}  \int_{\Sigma^{ext}_{\tau} }  \frac{(1+\tau )}{\eps} \cdot |  g^{\mu\a} \derm_{\mu } \derm_\a \Phi_{V}|^2    \cdot w(q) \cdot  d^{3}x  \cdot d\tau \\
 \notag
&& + \int_{t_1}^{t_2}  \int_{\Sigma^{ext}_{\tau} }   C(q_0) \cdot   c (\delta) \cdot c (\gamma) \cdot E (  4)  \cdot \frac{ \eps }{ (1+\tau)  } \cdot | \derm \Phi_{V} |^2     \cdot  w(q)\cdot  d^{3}x  \cdot d\tau \; . 
 \eeaa
 \end{lemma}
This is unlike the case of the Einstein vacuum equations, where an energy estimate for all components of the metric, without establishing an energy estimate for each component, would be sufficient. In fact, this is a key feature in this paper, that we have to deal separately at several occasions with some tangential components of the Einstein-Yang-Mills potential, namely $A_{e_a}$\;, before dealing with all components of the potential $A$\;, unlike the components of the metric which could be all dealt with together simultaneously at these occasions, which is a key difference not only with the Einstein vacuum equations, but also with the Einstein-Maxwell equations, where such “troublesome” structure does not exist -- we shall explain this later.

Such an energy estimate (see Lemma  \ref{Theveryfinal }), would allow us to control a weighted $L^2$ norm of $|  \derm A_{e_a} |$ (see the following estimate \eqref{ControlontheweightedLtwonormofdermAea}), by space-time integrals that do \textit{not} involve the “bad” term $ A_{e_a}  \cdot    \derm A_{e_a}$\;, because this “bad” term does not appear in the source terms for the wave equation for the “good” terms  $A_{e_a}$\;, $a \in \{1, 2 \}$\;. Indeed, for $\ga \geq 3 \de $\,, for $0 < \de \leq \frac{1}{4}$\,, and for $\eps$ small, enough depending on $q_0$\;, on $\ga$\;, on $\de$\;, on $|I|$ and on $\mu$\;, we have
      \bea\label{ControlontheweightedLtwonormofdermAea}
   \notag
 &&     \int_{\Sigma^{ext}_{t_2} }  |\derm ( \Lie_{Z^I}  A_{e_{a}} )  |^2     \cdot w(q)  \cdot d^{3}x \\  
    \notag
 &\les&      \int_{\Sigma^{ext}_{t_1} }  |\derm  ( \Lie_{Z^I}  A_{e_{a}} )|^2     \cdot w(q)  \cdot d^{3}x \\
    \notag
&& +   \sum_{|K| \leq |I |}   \int_{t_1}^{t_2}  \Big[    \;   \int_{\Sigma^{ext}_{\tau} }    \Big[       O \big(   C(q_0)   \cdot  c (\delta) \cdot c (\gamma) \cdot C(|I|) \cdot E ( \lfloor \frac{|I|}{2} \rfloor + 5)  \cdot \frac{\eps   \cdot  | \derm ( \Lie_{Z^K} h^1 ) |^2  }{(1+\tau+|q|)}   \big) \\
   \notag
   && + O \big(    C(q_0)   \cdot  c (\delta) \cdot c (\gamma) \cdot C(|I|) \cdot E (   \lfloor \frac{|I|}{2} \rfloor +5)  \cdot \frac{\eps  \cdot  |\derm ( \Lie_{Z^K} A ) |^2  }{(1+\tau+|q|) }  \big)    \; \Big]    \cdot w(q)  \cdot d^{3}x \Big]   \cdot d\tau  \\
      \notag
        &&  +     \int_{t_1}^{t_2}  \int_{\Sigma^{ext}_{\tau} }   \Big[     C(q_0)   \cdot  c (\delta) \cdot c (\gamma) \cdot C(|I|) \cdot E (  \lfloor \frac{|I|}{2} \rfloor   +3)  \cdot \frac{\eps }{(1+\tau+|q|)^{1-   c (\gamma)  \cdot c (\delta)  \cdot c(|I|) \cdot E (\lfloor \frac{|I|}{2} \rfloor  + 2)\cdot  \eps } \cdot (1+|q|)^2} \\
           \notag
&& \times    \sum_{  |K| \leq |I| -1 }  | \derm ( \Lie_{Z^K}  A )  |^2    \Big] \cdot w(q) \cdot d^{3}x  \cdot d\tau \\
   \notag
&& +   \int_{t_1}^{t_2}   \int_{\Sigma^{ext}_{\tau} }   \sum_{  |K| \leq |I| }        \Big[   C(q_0)   \cdot C(|I|) \cdot E (  \lfloor \frac{|I|}{2} \rfloor  +3)  \cdot \frac{\eps \cdot  | \Lie_{Z^{K}} H_{L  L} |^2 }{(1+\tau+|q|)^{1- 2 \de } \cdot (1+|q|)^{4+2\gamma}}  \Big] \cdot w(q)  \cdot d^{3}x  \cdot d\tau \\
 && +   C(q_0)   \cdot  c (\delta) \cdot c (\gamma) \cdot C(|I|) \cdot E ( \lfloor \frac{|I|}{2} \rfloor + 5)  \cdot  \frac{\eps^3 }{(1+t+|q|)^{1 -   c (\gamma)  \cdot c (\delta)  \cdot c(|I|) \cdot E ( \lfloor \frac{|I|}{2} \rfloor  + 5) \cdot \eps } }  \; .
\eea
This in turn, using a weighted Klainerman-Soboloev estimate in the exterior region (exterior to an outgoing null cone for the Minkowski metric, defined to be the Minkowski metric in wave coordinates), would translate into pointwise decay on $|  \derm A_{e_a} |$ (see the following estimate \eqref{upgradedestimateontheA_acomponentsthatconservesthedecayasoneoverttimesintegralsoftheenergysoastousegeneralizedgronwall}). Indeed, for $\ga \geq 3 \de $\,, for $0 < \de \leq \frac{1}{4}$\,, and for $\eps$ small, enough depending on $q_0$\,, on $\ga$, on $\de$, on $|I|$ and on $\mu < 0$\;, we have 
      \bea\label{upgradedestimateontheA_acomponentsthatconservesthedecayasoneoverttimesintegralsoftheenergysoastousegeneralizedgronwall}
   \notag
 &&   |\derm ( \Lie_{Z^I}  A_{e_{a}} )  |  \\
    \notag
 &\les&   \frac{1}{(1+t+|q|) \cdot (1+|q|)^{1+\ga} } \cdot \Big[ \;  \sum_{|J|\leq |I| + 2 }   \int_{\Sigma^{ext}_{t_1} }  |\derm  ( \Lie_{Z^J}  A_{e_{a}} )|^2     \cdot w(q)  \cdot d^{3}x \\
    \notag
&& +   \sum_{|K| \leq |I |}   \int_{t_1}^{t}  \Big[    \;   \int_{\Sigma^{ext}_{\tau} }    \Big[       O \big(   C(q_0)   \cdot  c (\delta) \cdot c (\gamma) \cdot C(|I|) \cdot E ( \lfloor \frac{|I|}{2} \rfloor + 6)  \cdot \frac{\eps   \cdot  | \derm ( \Lie_{Z^K} h^1 ) |^2  }{(1+\tau+|q|)}   \big) \\
   \notag
   && + O \big(    C(q_0)   \cdot  c (\delta) \cdot c (\gamma) \cdot C(|I|) \cdot E (   \lfloor \frac{|I|}{2} \rfloor +6)  \cdot \frac{\eps  \cdot  |\derm ( \Lie_{Z^K} A ) |^2  }{(1+\tau+|q|) }  \big)    \; \Big]    \cdot w(q)  \cdot d^{3}x  \Big]   \cdot d\tau  \\
      \notag
        &&  +     \int_{t_1}^{t}  \int_{\Sigma^{ext}_{\tau} }   \Big[     C(q_0)   \cdot  c (\delta) \cdot c (\gamma) \cdot C(|I|) \cdot E ( \lfloor \frac{|I|}{2} \rfloor  +4)  \cdot \frac{\eps }{(1+\tau+|q|)^{1-   c (\gamma)  \cdot c (\delta)  \cdot c(|I|) \cdot E (  \lfloor \frac{|I|}{2} \rfloor  + 4)\cdot  \eps } \cdot (1+|q|)^2} \\
           \notag
&& \times    \sum_{  |K| \leq |I| -1 }  | \derm ( \Lie_{Z^K}  A )  |^2     \Big]  \cdot w(q)  \cdot d^{3}x  \cdot d\tau \\
   \notag
&& +   \int_{t_1}^{t}   \int_{\Sigma^{ext}_{\tau} }   \sum_{  |K| \leq |I| }       \Big[   C(q_0)   \cdot C(|I|) \cdot E (  \lfloor \frac{|I|}{2} \rfloor  +5)  \cdot \frac{\eps \cdot  | \Lie_{Z^{K}} H_{L  L} |^2 }{(1+\tau+|q|)^{1- 2 \de } \cdot (1+|q|)^{4+2\gamma}}  \Big] \cdot w(q)  \cdot d^{3}x  \cdot d\tau   \\
 && +   C(q_0)   \cdot  c (\delta) \cdot c (\gamma) \cdot C(|I|) \cdot E ( \lfloor \frac{|I|}{2} \rfloor + 5)  \cdot  \frac{\eps^3 }{(1+t+|q|)^{1 -   c (\gamma)  \cdot c (\delta)  \cdot c(|I|) \cdot E ( \lfloor \frac{|I|}{2} \rfloor  + 5) \cdot \eps } }  \; \Big]^{\frac{1}{2}}  \; ,
\eea
and 
      \beaa
   \notag
 &&   |\Lie_{Z^I}  A_{e_{a}}   |  \\
 &\les&   \frac{1}{(1+t+|q|) \cdot (1+|q|)^{\ga} } \cdot \Big[ \;   \sum_{|J|\leq |I| + 2 }  \int_{\Sigma^{ext}_{t_1} }  |\derm  ( \Lie_{Z^J}  A_{e_{a}} )|^2     \cdot w(q)  \cdot d^{3}x \\
&& +   \sum_{|K| \leq |I |}   \int_{t_1}^{t}  \Big[    \;   \int_{\Sigma^{ext}_{\tau} }    \Big[       O \big(   C(q_0)   \cdot  c (\delta) \cdot c (\gamma) \cdot C(|I|) \cdot E ( \lfloor \frac{|I|}{2} \rfloor + 6)  \cdot \frac{\eps   \cdot  | \derm ( \Lie_{Z^K} h^1 ) |^2  }{(1+\tau+|q|)}   \big) \\
   && + O \big(    C(q_0)   \cdot  c (\delta) \cdot c (\gamma) \cdot C(|I|) \cdot E (   \lfloor \frac{|I|}{2} \rfloor +6)  \cdot \frac{\eps  \cdot  |\derm ( \Lie_{Z^K} A ) |^2  }{(1+\tau+|q|) }  \big)    \; \Big]    \cdot w(q)  \cdot d^{3}x  \Big]   \cdot d\tau  \\
      \notag
        &&  +     \int_{t_1}^{t}  \int_{\Sigma^{ext}_{\tau} }   \Big[     C(q_0)   \cdot  c (\delta) \cdot c (\gamma) \cdot C(|I|) \cdot E (\lfloor \frac{|I|}{2} \rfloor   +4)  \cdot \frac{\eps }{(1+\tau+|q|)^{1-   c (\gamma)  \cdot c (\delta)  \cdot c(|I|) \cdot E (  \lfloor \frac{|I|}{2} \rfloor  + 4)\cdot  \eps } \cdot (1+|q|)^2} \\
&& \times    \sum_{  |K| \leq |I| -1 }  | \derm ( \Lie_{Z^K}  A )  |^2    \Big]  \cdot w(q)  \cdot d^{3}x  \cdot d\tau \\
&& +   \int_{t_1}^{t}   \int_{\Sigma^{ext}_{\tau} }   \sum_{  |K| \leq |I| }       \Big[   C(q_0)   \cdot C(|I|) \cdot E (  \lfloor \frac{|I|}{2} \rfloor  +5)  \cdot \frac{\eps \cdot  | \Lie_{Z^{K}} H_{L  L} |^2 }{(1+\tau+|q|)^{1- 2 \de } \cdot (1+|q|)^{4+2\gamma}}  \Big] \cdot w(q)  \cdot d^{3}x  \cdot d\tau \\
 && +   C(q_0)   \cdot  c (\delta) \cdot c (\gamma) \cdot C(|I|) \cdot E ( \lfloor \frac{|I|}{2} \rfloor + 5)  \cdot  \frac{\eps^3 }{(1+t+|q|)^{1 -   c (\gamma)  \cdot c (\delta)  \cdot c(|I|) \cdot E ( \lfloor \frac{|I|}{2} \rfloor  + 5) \cdot \eps } }   \; \Big]^{\frac{1}{2}}  \; .
\eeaa
Hence, we used the Klainerman-Sobolev in the exterior to estimate this special component $|\derm ( \Lie_{Z^I}  A_{e_{a}} ) | $. Now, by using the very special fact that it is an $A_{e_a}$ component, and not just any component, it has the special feature that a partial derivative in the direction of $r$, is also a covariant derivative (covariant with respect to the Minkowski metric $m$) in the direction of $r$ of that component, i.e. $\pa_r  \Lie_{Z^I} A_{e_{a}}  = \derm_r  ( \Lie_{Z^I} A_{e_{a}} )$ (see the following equations \eqref{specialfactaboutAeacomponentsthatexpolitsthefactthatitisnotonlyAtangentialbutspecialone} and \eqref{partialderivativeindirectionofrofAeacomponenentiscovariantderivativeofAeacomponent}). In fact, in order to estimate first $ | \pa_{r} A_{e_{a}}  |$, so that we could then integrate along $s= constant$ and $\Om = constant$, we will use the special fact that
      \bea\label{specialfactaboutAeacomponentsthatexpolitsthefactthatitisnotonlyAtangentialbutspecialone}
   &&   \derm_{r} e_a = 0 \; ,
      \eea
 and therefore
      \bea\label{partialderivativeindirectionofrofAeacomponenentiscovariantderivativeofAeacomponent}
      \notag
    \pa_r  \Lie_{Z^I} A_{e_{a}} &=& \derm_r  ( \Lie_{Z^I} A_{e_{a}} )  +  \Lie_{Z^I} A (\derm_{r} e_a)  \\
     &=& \derm_r  ( \Lie_{Z^I} A_{e_{a}} ) \; .
     \eea
This in turn would allow us, by special integration, as we detailed in \cite{G4}, to translate the pointwise estimate on $|  \derm A_{e_a} |$ into pointwise estimate on $| A_{e_a} |$ (see estimate \eqref{upgradedestimateontheA_acomponentsthatconservesthedecayasoneoverttimesintegralsoftheenergysoastousegeneralizedgronwall}).

This would enable us to control the term $$\frac{ (1+t)}{\eps} \cdot \Big( \sum_{|K| +|J| \leq |I|}   | \Lie_{Z^K}  A_{e_{a}}   | \cdot |\derm ( \Lie_{Z^J}  A_{e_{a}} )  | \Big)^{2}$$ with the right decay factors in time and space (see the following estimate \eqref{estimateonthebadtermproductA_atimesdermA_ausingbootstrapsoastoclosethegronwallinequalityonenergy}). Indeed, for $\ga \geq 3 \de $\,, for $0 < \de \leq \frac{1}{4}$\,, and for $\eps$ small, enough depending on $q_0$\,, on $\ga$, on $\de$,  on $|I|$ and on $\mu < 0$\;, we have
  \bea\label{estimateonthebadtermproductA_atimesdermA_ausingbootstrapsoastoclosethegronwallinequalityonenergy}
\notag
&&  \frac{ (1+t)}{\eps} \cdot \Big( \sum_{|K| +|J| \leq |I|}   | \Lie_{Z^K}  A_{e_{a}}   | \cdot |\derm ( \Lie_{Z^J}  A_{e_{a}} )  | \Big)^{2}  \\
\notag
  &\les&   C(q_0) \cdot c (\gamma) \cdot  C ( |I| ) \cdot E (  \lfloor \frac{|I|}{2} \rfloor  + 2) \\
  \notag
  && \times  \frac{\eps}{(1+t+|q|)^{3-2\de} \cdot (1+|q|)^{2+4\ga} } \cdot \Big[ \;  \sum_{|J|\leq |I| + 2 }   \int_{\Sigma^{ext}_{t_1} }  |\derm  ( \Lie_{Z^J}  A_{e_{a}} )|^2     \cdot w(q)  \cdot d^{3}x \\
  \notag
&& +   \sum_{|K| \leq |I |}   \int_{t_1}^{t}  \Big[    \;   \int_{\Sigma^{ext}_{\tau} }    \Big[       O \big(   C(q_0)   \cdot  c (\delta) \cdot c (\gamma) \cdot C(|I|) \cdot E ( \lfloor \frac{|I|}{2} \rfloor + 6)  \cdot \frac{\eps   \cdot  | \derm ( \Lie_{Z^K} h ) |^2  }{(1+\tau+|q|)}   \big) \\
\notag
   && + O \big(    C(q_0)   \cdot  c (\delta) \cdot c (\gamma) \cdot C(|I|) \cdot E (   \lfloor \frac{|I|}{2} \rfloor +6)  \cdot \frac{\eps  \cdot  |\derm ( \Lie_{Z^K} A ) |^2  }{(1+\tau+|q|) }  \big)    \; \Big]    \cdot w(q)  \cdot d^{3}x   \Big]  \cdot d\tau  \\
      \notag
        &&  +     \int_{t_1}^{t}  \int_{\Sigma^{ext}_{\tau} }   \Big[     C(q_0)   \cdot  c (\delta) \cdot c (\gamma) \cdot C(|I|) \cdot E (\lfloor \frac{|I|}{2} \rfloor   +4)  \cdot \frac{\eps }{(1+\tau+|q|)^{1-   c (\gamma)  \cdot c (\delta)  \cdot c(|I|) \cdot E (  \lfloor \frac{|I|}{2} \rfloor + 4)\cdot  \eps } \cdot (1+|q|)^2} \\
        \notag
&& \times    \sum_{  |K| \leq |I| -1 }  | \derm ( \Lie_{Z^K}  A )  |^2    \Big]  \cdot w(q)  \cdot d^{3}x  \cdot d\tau  \\
\notag
&& +   \int_{t_1}^{t}   \int_{\Sigma^{ext}_{\tau} }   \sum_{  |K| \leq |I| }        \Big[   C(q_0)   \cdot C(|I|) \cdot E (  \lfloor \frac{|I|}{2} \rfloor  +5)  \cdot \frac{\eps \cdot  | \Lie_{Z^{K}} H_{L  L} |^2 }{(1+\tau+|q|)^{1- 2 \de } \cdot (1+|q|)^{4+2\gamma}}  \Big] \cdot w(q)  \cdot d^{3}x  \cdot d\tau  \\
 && +   C(q_0)   \cdot  c (\delta) \cdot c (\gamma) \cdot C(|I|) \cdot E ( \lfloor \frac{|I|}{2} \rfloor + 5)  \cdot  \frac{\eps^3 }{(1+t+|q|)^{1 -   c (\gamma)  \cdot c (\delta)  \cdot c(|I|) \cdot E ( \lfloor \frac{|I|}{2} \rfloor  + 5) \cdot \eps } }  \; \Big]  \; .
\eea
This estimate is done so that the \textit{weighted} integral in space of such quantity, namely $$\int_{\Sigma^{ext}_{\tau} }  \frac{ (1+t)}{\eps} \cdot \Big( \sum_{|K| +|J| \leq |I|}   | \Lie_{Z^K}  A_{e_{a}}   | \cdot |\derm ( \Lie_{Z^J}  A_{e_{a}} )  | \Big)^{2}   \cdot w(q) \; ,$$ could be controlled (see the following estimate \eqref{HardyinequalityontheestimateonthebadtermproductA_atimesdermA_ausingbootstrapsoastoclosethegronwallinequalityonenergy}) by a quantity that it is a time integral multiplied with the right factor so that the time integral $$\int_{t_1}^{t_2} \int_{\Sigma^{ext}_{\tau} }  \frac{ (1+t)}{\eps} \cdot \Big( \sum_{|K| +|J| \leq |I|}   | \Lie_{Z^K}  A_{e_{a}}   | \cdot |\derm ( \Lie_{Z^J}  A_{e_{a}} )  | \Big)^{2}   \cdot w(q) \cdot dt $$ would be controlled by a time integral of a time integral, with all terms entering with the right factors so that they could be controlled only by one single time integral, that gives a suitable control to be able to apply later a Grönwall inequality on all components. For $\ga \geq 3 \de $\,, for $0 < \de \leq \frac{1}{4}$\,, and for $\eps$ small, enough depending on $q_0$\;, on $\ga$\;, on $\de$\;, on $|I|$ and on $\mu$\;, we have
 \bea\label{HardyinequalityontheestimateonthebadtermproductA_atimesdermA_ausingbootstrapsoastoclosethegronwallinequalityonenergy}
\notag
&&   \int_{\Sigma^{ext}_{\tau} }  \frac{ (1+t)}{\eps} \cdot \Big( \sum_{|K| +|J| \leq |I|}   | \Lie_{Z^K}  A_{e_{a}}   | \cdot |\derm ( \Lie_{Z^J}  A_{e_{a}} )  | \Big)^{2}   \cdot w(q)  \\
\notag
  &\les&   C(q_0) \cdot c (\gamma) \cdot  C ( |I| ) \cdot E (  \lfloor \frac{|I|}{2} \rfloor  + 2) \\
  \notag
  && \times  \frac{\eps}{(1+t)^{1+2\de}}  \cdot \Big[ \;  \sum_{|J|\leq |I| + 2 }   \int_{\Sigma^{ext}_{t_1} }  |\derm  ( \Lie_{Z^J}  A_{e_{a}} )|^2     \cdot w(q)  \cdot d^{3}x \\
  \notag
&& +   \sum_{|K| \leq |I |}   \int_{t_1}^{t}  \Big[    \;   \int_{\Sigma^{ext}_{\tau} }    \Big[       O \big(   C(q_0)   \cdot  c (\delta) \cdot c (\gamma) \cdot C(|I|) \cdot E ( \lfloor \frac{|I|}{2} \rfloor + 6)  \cdot \frac{\eps   \cdot  | \derm ( \Lie_{Z^K} h^1 ) |^2  }{(1+\tau+|q|)}   \big) \\
\notag
   && + O \big(    C(q_0)   \cdot  c (\delta) \cdot c (\gamma) \cdot C(|I|) \cdot E (   \lfloor \frac{|I|}{2} \rfloor +6)  \cdot \frac{\eps  \cdot  |\derm ( \Lie_{Z^K} A ) |^2  }{(1+\tau+|q|) }  \big)    \; \Big]    \cdot w(q)  \cdot d^{3}x   \Big]  \cdot d\tau  \\
      \notag
        &&  +     \int_{t_1}^{t}  \int_{\Sigma^{ext}_{\tau} }   \Big[     C(q_0)   \cdot  c (\delta) \cdot c (\gamma) \cdot C(|I|) \cdot E (\lfloor \frac{|I|}{2} \rfloor   +4)  \cdot \frac{\eps }{(1+\tau+|q|)^{1-   c (\gamma)  \cdot c (\delta)  \cdot c(|I|) \cdot E (  \lfloor \frac{|I|}{2} \rfloor + 4)\cdot  \eps } \cdot (1+|q|)^2} \\
        \notag
&& \times    \sum_{  |K| \leq |I| -1 }  | \derm ( \Lie_{Z^K}  A )  |^2     \Big]  \cdot w(q)  \cdot d^{3}x  \cdot d\tau  \\
\notag
&& +   \int_{t_1}^{t_2}   \int_{\Sigma^{ext}_{\tau} }   \sum_{  |K| \leq |I| }         \Big[   C(q_0)   \cdot C(|I|) \cdot E (  \lfloor \frac{|I|}{2} \rfloor  +5)  \cdot \frac{\eps \cdot  | \Lie_{Z^{K}} H_{L  L} |^2 }{(1+\tau+|q|)^{1- 2 \de } \cdot (1+|q|)^{4+2\gamma}}  \Big]  \cdot w(q)  \cdot d^{3}x  \cdot d\tau \\
 && +   C(q_0)   \cdot  c (\delta) \cdot c (\gamma) \cdot C(|I|) \cdot E ( \lfloor \frac{|I|}{2} \rfloor + 5)  \cdot  \frac{\eps^3 }{(1+t)^{1 -   c (\gamma)  \cdot c (\delta)  \cdot c(|I|) \cdot E ( \lfloor \frac{|I|}{2} \rfloor  + 5) \cdot \eps } }   \; \Big]     \; .
\eea

\subsection{Dealing with the “bad” term $ A_L   \cdot   \derm A $\;}\label{dealingwithALderA}\

It turns out that the Lorenz gauge condition, implies a very good estimate in the exterior on the zeroth Lie derivative of $A_L$ (see the following estimate \eqref{simpleestimateonlyonthe ALcomponenetwithOUTLiederivative}, or see also the following estimate \eqref{estimategoodcomponentspotentialandmetric}, where the sum over $|I|-1$ is absent if $|I| = 0$). Indeed, under the bootstrap assumption holding for all $|I| \leq 3 $, we have in the exterior region,
          \bea\label{simpleestimateonlyonthe ALcomponenetwithOUTLiederivative}
        \notag
 |  A_{L}  |   &\les&  C(q_0)   \cdot  c (\delta) \cdot c (\gamma)  \cdot E ( 3) \cdot \frac{ \eps   }{ (1+t+|q|)^{2-2\delta}  \cdot  (1+|q|)^{\ga - 1}  } \; .\\
       \eea
Let $M\leq \eps$\,. Under the bootstrap assumption holding for all $|J| \leq  \lfloor \frac{|I|}{2} \rfloor  $\,, we have
        \bea\label{estimategoodcomponentspotentialandmetric}
        \notag
 |   \Lie_{Z^I}  A_{L}  |  (t, | x | \cdot \Om)  &\leq& \int\limits_{s,\,\Om=const} \sum_{|J|\leq |I| -1} |   \derm (  \Lie_{Z^J} A )  | \\
 \notag
     && + \begin{cases} c (\delta) \cdot c (\gamma) \cdot C ( |I| ) \cdot E ( |I| + 3)  \cdot \big( \frac{ \eps   }{ (1+t+|q|)^{2-2\delta} \cdot  (1+|q|)^{\ga - 1} } \big),\quad\text{when }\quad q>0 \; ,\\
       C ( |I| ) \cdot E ( |I| + 3)  \cdot \big( \frac{ \eps \cdot (1+|q|)^{\frac{3}{2} }   }{ (1+t+|q|)^{2-2\delta} } \big) \,\quad\text{when }\quad q<0 \; , \end{cases} \\
       \eea
        where $\int\limits_{s,\, \Om=const}  |   \derm  ( \Lie_{Z^J} A )  |   $ is the integral defined as in the following Definition \ref{definitionforintegralalongthenullcoordinateplusboundaryterm}.   
\begin{definition}\label{definitionforintegralalongthenullcoordinateplusboundaryterm}
For a function $f$, we define $\int\limits_{s,\, \Om=const} f  $ as the integral at a fixed $\Om \in \SSS^2$, from $(t, | x | \cdot \Om)$ along the line $(\tau, r \cdot \Om)$ such that $r+\tau =  | x | +t $ (i.e. along a fixed null coordinate $s:=\tau+r $) till we reach the hyperplane $\tau=0$, to which we also add the generated boundary term at the hyperplane prescribed by $\tau=0$. In other words,
\bea\label{definitionequationforintegralalongthenullcoordinateplusboundaryterm}
\notag
\int\limits_{s,\, \Om=const} |f |  (t, | x | \cdot \Om)  &=&   \int_{t+| x | }^{| x |  } \pa_r | f (t + | x | - r,  r  \cdot \Om ) | dr +  | f \big(0, ( t + | x |) \cdot \Om \big) |  \; .\\
\eea
\end{definition}
Also, under the bootstrap assumption holding for all $|J| \leq  \lfloor \frac{|I|}{2} \rfloor  $,
        \beaa
        \notag
|  \Lie_{Z^I} h_{\cal T L} |  + |  \Lie_{Z^I} H_{\cal T L} |   &\les&  \int\limits_{s,\,\Om=const} \sum_{|J|\leq |I| -1} \big( |\derm  ( \Lie_{Z^J} h) | + |\derm  ( \Lie_{Z^J} H )  | \big) \\
&& + \begin{cases} c (\delta) \cdot c (\gamma) \cdot C ( |I| ) \cdot E ( |I| + 3)  \cdot \frac{ \eps   }{ (1+t+|q|) } ,\quad\text{when }\quad q>0 \; ,\\
        C ( |I| ) \cdot E ( |I| + 3) \cdot  \frac{ \eps \cdot (1+|q|)^{\frac{1}{2} + 2 \delta } }{ (1+t+|q|) }  \,\quad\text{when }\quad q<0 \; . \end{cases} 
       \eeaa
   and
           \beaa
        \notag
|  \Lie_{Z^I} h_{L L} |  + |  \Lie_{Z^I} H_{ L L} |   &\les&  \int\limits_{s,\,\Om=const} \sum_{|J|\leq |I| -2} \big( |\derm  ( \Lie_{Z^J} h ) | + |\derm  ( \Lie_{Z^J} H ) | \big) \\
&& + \begin{cases} c (\delta) \cdot c (\gamma) \cdot C ( |I| ) \cdot E ( |I| + 3)  \cdot \frac{ \eps   }{ (1+t+|q|) } ,\quad\text{when }\quad q>0 \; ,\\
        C ( |I| ) \cdot E ( |I| + 3) \cdot  \frac{ \eps \cdot (1+|q|)^{\frac{1}{2} + 2 \delta } }{ (1+t+|q|) }  \,\quad\text{when }\quad q<0 \; . \end{cases} 
       \eeaa    
As a result, more precisely, in $\overline{C} \subset \{ q \geq q_0 \} $, we have 
                    \beaa
        \notag
|  \Lie_{Z^I} h_{\cal T L} |    + |  \Lie_{Z^I} H_{\cal T L} |                    &\les&   \int\limits_{s,\,\Om=const} \sum_{|J|\leq |I| -1} \big( |\derm  ( \Lie_{Z^J} h) | + |\derm (  \Lie_{Z^J} H  ) | \big) \\
\notag
&& + C(q_0) \cdot   c (\delta) \cdot c (\gamma) \cdot C ( |I| ) \cdot E ( |I| + 3)  \cdot \frac{ \eps   }{ (1+t+|q|) } \; , 
       \eeaa
and
                \beaa
        \notag
|  \Lie_{Z^I} h_{L L} |    + |  \Lie_{Z^I} H_{L L} |                    &\les&   \int\limits_{s,\,\Om=const} \sum_{|J|\leq |I| -2} \big( |\derm  ( \Lie_{Z^J} h) | + |\derm  ( \Lie_{Z^J} H )  | \big) \\
\notag
&& + C(q_0) \cdot   c (\delta) \cdot c (\gamma) \cdot C ( |I| ) \cdot E ( |I| + 3)  \cdot \frac{ \eps   }{ (1+t+|q|) } \; . 
       \eeaa 
We also have a very good estimate on the partial derivative of that component $\pa A_L$ (thanks to the following estimate \eqref{estimateonpartialderivativeofALcomponent}),  which in an estimate also needed to control successfully the term $ | A_L  | \cdot    | \derm A   | $\;. Indeed, under the bootstrap assumption \eqref{aprioriestimate} holding for all $|J| \leq  \lfloor \frac{|I|}{2} \rfloor  $\;, we have the following estimate for the Einstein-Yang-Mills potential, 
\bea\label{estimateonpartialderivativeofALcomponent}
\notag
  && | \pa  \Lie_{Z^I}  A_{L}  |  \\
  \notag
  &\les& E(  \lfloor \frac{|I|}{2} \rfloor   )  \cdot \Big(  \sum_{|J|\leq |I|} |   \rderm  \Lie_{Z^J} A | + \sum_{|J|\leq |I| -1} |   \derm  \Lie_{Z^J} A |  + \sum_{|K| + |M| \leq |I|} O \big(  | ( \Lie_{Z^K}  h )  | \cdot  | \derm  (  \Lie_{Z^M}  A ) | \big) \Big)  \; ,\\
\eea
where 
\beaa
| \pa  \Lie_{Z^I}  A_{L} |^2 = | \pa_t ( \Lie_{Z^I} A )_{L} |^2 + \sum_{i=1}^3 | \pa_i ( \Lie_{Z^I} A)_{L} |^2 \;. 
\eeaa
Also, under the bootstrap assumption \eqref{aprioriestimate} holding for all $|J| \leq  \lfloor \frac{|I|}{2} \rfloor  $\;, we have the following estimate for the Einstein-Yang-Mills metric, 
\bea\label{partialderivativeofLiederivativeofHTangentialLcompo}
  \notag
  && | \pa  \Lie_{Z^I}  H_{\cal T L}  | \\
    \notag
   &\les& E(  \lfloor \frac{|I|}{2} \rfloor   )  \cdot \Big(   \sum_{|J|\leq |I|} |   \rderm  \Lie_{Z^J} H | + \sum_{|J|\leq |I| -1} |   \derm  \Lie_{Z^J} H |  + \sum_{|K|+ |M| \leq |I|}  O (|\Lie_{Z^K} H| \cdot |\derm ( \Lie_{Z^M} H )| )   \Big)  \; ,\\
\eea
\bea\label{partialderivativeofLiederivativeofHLLcomp}
  \notag
  && | \pa  \Lie_{Z^I}  H_{ L L}  |  \\
    \notag
    &\les& E(  \lfloor \frac{|I|}{2} \rfloor   )  \cdot \Big(  \sum_{|J|\leq |I|} |   \rderm  \Lie_{Z^J} H | + \sum_{|J|\leq |I| -2} |   \derm  \Lie_{Z^J} H |  + \sum_{|K|+ |M| \leq |I|}  O (|\Lie_{Z^K} H| \cdot |\derm ( \Lie_{Z^M} H )| )   \Big)  \; ,\\
\eea
and 
\bea\label{partialderivativeofLiederivativeofsmallhTangentialLcompo}
  \notag
  && | \pa  \Lie_{Z^I}  h_{\cal T L}  | \\
    \notag
   &\les& E(  \lfloor \frac{|I|}{2} \rfloor   )  \cdot \Big(  \sum_{|J|\leq |I|} |   \rderm  \Lie_{Z^J} h | + \sum_{|J|\leq |I| -1} |   \derm  \Lie_{Z^J} h |  + \sum_{|K|+ |M| \leq |I|}  O (|\Lie_{Z^K} h| \cdot |\derm ( \Lie_{Z^M} h )| )   \Big)  \; ,\\
\eea
\bea\label{partialderivativeofLiederivativeofsmallhLLcomp}
  \notag
  && | \pa  \Lie_{Z^I}  h_{ L L}  |  \\
  \notag
  &\les& E(  \lfloor \frac{|I|}{2} \rfloor   )  \cdot \Big(  \sum_{|J|\leq |I|} |   \rderm  \Lie_{Z^J} h | + \sum_{|J|\leq |I| -2} |   \derm  \Lie_{Z^J} h |  + \sum_{|K|+ |M| \leq |I|}  O (|\Lie_{Z^K} h| \cdot |\derm ( \Lie_{Z^M} h )| )   \Big)  \; ,\\
\eea
where 
\beaa
| \pa  \Lie_{Z^I}  H_{\cal T L}  |^2 = | \pa_t ( \Lie_{Z^I} H  )_{\cal T L} |^2 + \sum_{i=1}^3 | \pa_i ( \Lie_{Z^I} H  )_{\cal T L} |^2 \;. 
\eeaa
In these estimates the partial differentiation actually concerns a partial derivative of the stated components of the Lie derivatives of the tensor (as opposed to the covariant derivative), and where the $\sum_{|J|\leq  -1}$ is understood to be a vanishing sum. In fact, the leading term of Lie derivatives in the direction of Minkowski vector fields of this product, namely $$ \sum_{|K| = |I| }  \Big(  |  A_L |  \cdot     |  \derm ( \Lie_{Z^K} A )  |  +  |  \Lie_{Z^K} A_L |  \cdot     |  \derm A   |   \Big) \;,$$ contains on one hand, a term  $ \sum_{|K| = |I| }  |  \Lie_{Z^K} A_L | $ that enters with a factor that seems to be the “wrong” factor (see the following Corollary \ref{TheleeadingtermfortroublesomecomponentsintehsourcesforA}), and on the other hand, a term $ \sum_{|K| = |I| } |  \derm ( \Lie_{Z^K} A )  | $ that enters with the right factor in order to apply a Grönwall inequality on the energy. 
\begin{corollary} \label{TheleeadingtermfortroublesomecomponentsintehsourcesforA}
We have for $\ga \geq 3 \de $\,, and $0 < \de \leq \frac{1}{4}$,
\beaa
&& \sum_{|K| = |I| }  \Big(  |  A_L |  \cdot     |  \derm ( \Lie_{Z^K} A )  |  +  |  \Lie_{Z^K} A_L |  \cdot     |  \derm A   |   \Big) \\
&\les& \sum_{|K| = |I| }  \Big( C(q_0)   \cdot  c (\delta) \cdot c (\gamma)  \cdot E ( 3) \cdot \frac{ \eps  \cdot     |  \derm ( \Lie_{Z^K} A )  |   }{ (1+t+|q|)^{2-2\delta}  \cdot  (1+|q|)^{\ga - 1}  } \\
&&  +    C(q_0)   \cdot  c (\delta) \cdot c (\gamma)  \cdot E ( 4)  \cdot \frac{\eps \cdot   |  \Lie_{Z^K} A_L |   }{(1+t+|q|)^{1-      c (\gamma)  \cdot c (\delta)   \cdot E ( 4) \cdot \eps } \cdot (1+|q|)^{1+\gamma - 2\de }}     \Big) 
\eeaa
\end{corollary}
However, applying a Hardy type inequality in the exterior (see the following Corollary \ref{HardytypeinequalityforintegralstartingatROm} of a Hardy type inequality), we can estimate the term with the “wrong” factor, namely $$ \sum_{|K| = |I| }  |  \Lie_{Z^K} A_L | \; , $$ by its partial derivative instead, namely $| \pa  \Lie_{Z^K} A_L  | $\;.
\begin{corollary}\label{HardytypeinequalityforintegralstartingatROm}
Let $w$ defined as in Definition \ref{defoftheweightw}, where $\ga > 0$.
Let  $\Phi$ a tensor that decays fast enough at spatial infinity for all time $t$\,, such that
\bea
 \int_{\SSS^{2}} \lim_{r \to \infty} \Big( \frac{r^{2}}{(1+t+r)^{a} \cdot (1+|q|) } \cdot w(q) \cdot | \Phi_{V} |^2   \Big)  d\si^{2} (t ) &=& 0 \; .
\eea
Let $R(\Om)  \geq 0 $\,, be a function of $\Om \in \SSS^{n-1}$\,. Then, since $\ga \neq 0$\,, we have for $0 \leq a \leq 2$\,, that 
\bea
\notag
 &&   \int_{\SSS^{2}} \int_{r=R(\Om)}^{r=\infty} \frac{r^{2}}{(1+t+r)^{a}} \cdot   \frac{w (q)}{(1+|q|)^2} \cdot  |\Phi_{V} |^2   \cdot dr  \cdot d\si^{2}  \\
 \notag
 &\leq& c(\ga) \cdot  \int_{\SSS^{n-1}} \int_{r=R(\Om)}^{r=\infty}  \frac{ r^{2}}{(1+t+r)^{a}} \cdot w(q) \cdot  | \pa_r\Phi_{V} |^2  \cdot  dr  \cdot d\si^{2}  \; , \\
 \eea
 where the constant $c(\ga)$ does not depend on $R(\Om)$\,. In particular, we have the following estimate in the exterior,
 \bea
\notag
  \int_{\Sigma^{ext}_{\tau} }  \frac{1}{(1+t+r)^{a}} \cdot   \frac{w (q)}{(1+|q|)^2} \cdot  |\Phi_{V} |^2 &\leq& c(\ga) \cdot  \int_{\Sigma^{ext}_{\tau} }   \frac{ 1 }{(1+t+r)^{a}} \cdot w(q) \cdot  | \pa_r\Phi_{V} |^2     \; , \\
 \eea
  where $\int_{\Sigma^{ext}_{\tau} } $ is taken with respect to the measure $r^2 \cdot dr  \cdot d\si^{2} $\,. And, in particular, if 
 \beaa
 \int_{\SSS^{2}} \lim_{r \to \infty} \Big( \frac{r^{2}}{(1+t+r)^{a} \cdot (1+|q|) } \cdot w(q) \cdot | \Phi |^2   \Big)  d\si^{2} (t ) &=& 0 \; .
\eeaa
then,
  \bea
\notag
 \int_{\Sigma^{ext}_{\tau} }  \frac{1}{(1+t+r)^{a}} \cdot   \frac{w (q)}{(1+|q|)^2} \cdot  |\Phi |^2  &\leq& c(\ga) \cdot  \int_{\Sigma^{ext}_{\tau} }   \frac{ 1 }{(1+t+r)^{a}} \cdot w(q) \cdot  | \derm \Phi |^2    \; . \\
 \eea
 \end{corollary}
In the above estimate, we have again a seemingly “wrong” factor, but thanks to estimate \eqref{estimateonpartialderivativeofALcomponent}, this partial derivative can then be estimated by tangential derivatives of the potential $|\rderm A|$\;, with some good error terms, for which we have a very good control on (in the following estimate \eqref{estimateonthetermthatcontainspartialderivatoveofALinthettimessquareofbadcomponenentsinsourcetermsforAthatcontainstheALcomponent}), despite the “wrong” decaying factor for $|\rderm A|$\;, and this is thanks to our energy estimate (see Lemma \ref{Theveryfinal }) that controls a space-time integral for such tangential derivatives. For $\ga \geq 3\de$, and for $\eps$ small depending on $|I|$\;, on $\ga$ and on $\de$\;, we have
\bea\label{estimateonthetermthatcontainspartialderivatoveofALinthettimessquareofbadcomponenentsinsourcetermsforAthatcontainstheALcomponent}
\notag
 && \int_{\Sigma^{ext}_{\tau} } C(q_0)   \cdot  c (\delta) \cdot c (\gamma)  \cdot E ( 4)   \cdot \sum_{|K| = |I| }  \Big(    \frac{\eps \cdot   | \pa \Lie_{Z^K} A_L |^2   }{(1+t+|q|)^{1-      c (\gamma)  \cdot c (\delta)   \cdot E ( 4) \cdot \eps } \cdot (1+|q|)^{2\gamma - 4\de }}       \Big)  \cdot w(q)   \\
 \notag
 &\les & \int_{\Sigma^{ext}_{\tau} }C(q_0)   \cdot  c (\delta) \cdot c (\gamma)    \cdot  E (\lfloor \frac{|I|}{2} \rfloor +4  ) \cdot  \\
 \notag
 && \times   \Big[  \sum_{|K| \leq |I| }     \frac{\eps \cdot   | \rderm ( \Lie_{Z^K} A) |^2   }{(1+t+|q|)^{1-      c (\gamma)  \cdot c (\delta)   \cdot E ( 4) \cdot \eps } \cdot (1+|q|)^{2\gamma - 4\de }}         \\
 \notag
 &&+  \sum_{|K| \leq |I|-1 }     \frac{\eps \cdot   | \derm \Lie_{Z^K} A |^2   }{(1+t+|q|)^{1-      c (\gamma)  \cdot c (\delta)   \cdot E ( 4) \cdot \eps } \cdot (1+|q|)^{2\gamma - 4\de }}       \Big]  \cdot w(q)   \\
 \notag
&&+   \int_{\Sigma^{ext}_{\tau} } C(q_0)   \cdot  c (\delta) \cdot c (\gamma) \cdot C(|I|) \cdot E (  \lfloor \frac{|I|}{2} \rfloor  +4)  \\
\notag
&&  \times  \frac{\eps }{(1+t+|q|)^{2 } }  \cdot  \sum_{|K| \leq |I|}  \Big[  | \derm  (  \Lie_{Z^K}  A ) |^2 +  |   \derm ( \Lie_{Z^K} h^1 )   |^2  \Big]  \cdot w(q) \\
&&+   C(q_0)   \cdot  c (\delta) \cdot c (\gamma) \cdot C(|I|) \cdot E (  \lfloor \frac{|I|}{2} \rfloor  +4)  \cdot  \frac{\eps^3 }{(1+t)^{1-     c (\gamma)  \cdot c (\delta)  \cdot c(|I|) \cdot E ( \lfloor \frac{|I|}{2} \rfloor+ 4) \cdot \eps } }      \; .
\eea
The lower order terms, lower in number of Lie derivatives, namely $$\sum_{|J| + |K| \leq |I| -1}    | \Lie_{Z^J}  A_L |  \cdot     |  \derm ( \Lie_{Z^K} A )  | \;, $$ have a factor for $  |  \derm ( \Lie_{Z^K} A )  |$ which this time involves the Lie derivatives of that special component, namely $|\Lie_{Z^K}  A_L |$\;, and not the “nice” zeroth Lie derivative as opposed to what we explained above. To control the term $  |  \derm ( \Lie_{Z^K} A )  |$ that enters with the “wrong” factor, we use the fact that we have only an $\eps$ loss in the time decay rate (see the following estimate \eqref{estimateonlowerordertermsforthepotentialAinthesourcetermsforenergy}) thanks to our upgrading of the dispersive estimates for the Lie derivatives of the fields (see Subsection \ref{dealingwithLiederivativesupgrade}, that also presents new challenges that we shall explain later). Indeed, we have for $\ga \geq 3 \de $\,, and $0 < \de \leq \frac{1}{4}$\,,
               \bea\label{estimateonlowerordertermsforthepotentialAinthesourcetermsforenergy}
               \notag
&& \sum_{|J| + |K| \leq |I| -1}    | \Lie_{Z^J}  A_L |  \cdot     |  \derm ( \Lie_{Z^K} A )  |     \\
               \notag
&\les& C(q_0)   \cdot  c (\delta) \cdot c (\gamma) \cdot C( \lfloor \frac{|I|-1}{2} \rfloor) \cdot E ( \lfloor \frac{|I|-1}{2} \rfloor + 4) \\
               \notag
&& \times \sum_{ |K| \leq |I| -1}    \Big(    \frac{\eps  \cdot     |  \derm ( \Lie_{Z^K} A )  | }{(1+t+|q|)^{1-      c (\gamma)  \cdot c (\delta)  \cdot c( \lfloor \frac{|I|-1}{2} \rfloor) \cdot E ( \lfloor \frac{|I|-1}{2} \rfloor+ 4) \cdot \eps } \cdot (1+|q|)^{\gamma - 2\de }}    \\
&& + \frac{\eps \cdot  | \Lie_{Z^K}  A |  }{(1+t+|q|)^{1-      c (\gamma)  \cdot c (\delta)  \cdot c(\lfloor \frac{|I|-1}{2} \rfloor ) \cdot E (\lfloor \frac{|I|-1}{2} \rfloor + 4) \cdot \eps } \cdot (1+|q|)^{1+\gamma - 2\de }}   \Big) \;. 
\eea       
To control the lower order terms $| \Lie_{Z^J}  A_L |$\;, we plan carefully to use again, in the following Lemma \ref{HardytpyeestimatesontehtermscontainingALinthebadstructureofsourcestermsforAusefultoobtainenergyestimates}, the Hardy type inequality in the exterior (given in Corollary \ref{HardytypeinequalityforintegralstartingatROm}) to translate them into partial derivatives, but this time, we do not care any more if these are partial derivative of special components. In fact, since all of these are lower order terms, we can close the argument by choosing $\eps$ small enough, small depending on the number of derivatives that we want to control (and depending on other parameters that we shall all show so that one can follow carefully our argument) -- and obviously we must control at least some derivatives so that our applications of the weighted Klainerman-Sobolev inequalities would make sense.
\begin{lemma}\label{HardytpyeestimatesontehtermscontainingALinthebadstructureofsourcestermsforAusefultoobtainenergyestimates}
We have for $\ga \geq 3 \de $\,, and $0 < \de \leq \frac{1}{4}$\,,
                                         \beaa
\notag
&&  \int_{\Sigma^{ext}_{\tau} }  \frac{ (1+t)}{\eps}  \cdot \Big(  \sum_{|K| \leq |I|}  | \Lie_{Z^K}    \big(  A_L   \cdot     \derm A    \big)    |   \Big)^{2}  \cdot w(q)  \\
&\les&   \int_{\Sigma^{ext}_{\tau} } \Big[ C(q_0)   \cdot  c (\delta) \cdot c (\gamma)  \cdot E ( 4)   \\
&& \times \sum_{|K| = |I| }  \Big(  \frac{ \eps  \cdot     |  \derm ( \Lie_{Z^K} A )  |^2   }{ (1+t+|q|)  \cdot  (1+|q|)^{2\gamma - 4\de}  }   +    \frac{\eps \cdot   | \pa \Lie_{Z^K} A_L |^2   }{(1+t+|q|)^{1-      c (\gamma)  \cdot c (\delta)   \cdot E ( 4) \cdot \eps } \cdot (1+|q|)^{2\gamma - 4\de }}       \Big) \; \Big]  \cdot w(q)  \\
&& +   \int_{\Sigma^{ext}_{\tau} } \Big[  C(q_0)   \cdot  c (\delta) \cdot c (\gamma) \cdot C( |I|) \cdot E ( \lfloor \frac{|I|-1}{2} \rfloor + 4) \\
&& \times \sum_{ |K| \leq |I| -1}    \Big(    \frac{\eps  \cdot     |  \derm ( \Lie_{Z^K} A )  |^2 }{(1+t+|q|)^{1-      c (\gamma)  \cdot c (\delta)  \cdot c( \lfloor \frac{|I|-1}{2} \rfloor) \cdot E ( \lfloor \frac{|I|-1}{2} \rfloor+ 4) \cdot \eps } \cdot (1+|q|)^{2\gamma - 4\de }}      \Big)   \; \Big]   \cdot w(q)  \;. 
\eeaa   
\end{lemma}
\begin{proof}
We estimate the term $ | \Lie_{Z^I}    \big(  A_L   \cdot     \derm A    \big)    | $ in a way that gives a suitable control in the following estimate \eqref{thestructureofthebadtermALdermAusingboostrapassumptionanddecompistionofthesumandlowerordertermsexhibitedtodealwithAL}. Indeed, we have for $\ga \geq 3 \de $\,, and $0 < \de \leq \frac{1}{4}$\,.
                                         \bea\label{thestructureofthebadtermALdermAusingboostrapassumptionanddecompistionofthesumandlowerordertermsexhibitedtodealwithAL}
\notag
&& | \Lie_{Z^I}    \big(  A_L   \cdot     \derm A    \big)    | \\
\notag
&\les&  C(q_0)   \cdot  c (\delta) \cdot c (\gamma)  \cdot E ( 4)  \\
\notag
&& \times \sum_{|K| = |I| }  \Big(  \frac{ \eps  \cdot     |  \derm ( \Lie_{Z^K} A )  |   }{ (1+t+|q|)  \cdot  (1+|q|)^{\gamma - 2\de}  } \\
\notag
&&  +    \frac{\eps \cdot   |  \Lie_{Z^K} A_L |   }{(1+t+|q|)^{1-      c (\gamma)  \cdot c (\delta)   \cdot E ( 4) \cdot \eps } \cdot (1+|q|)^{1+\gamma - 2\de }}    \Big) \\
\notag
&& +  C(q_0)   \cdot  c (\delta) \cdot c (\gamma) \cdot C( \lfloor \frac{|I|-1}{2} \rfloor) \cdot E ( \lfloor \frac{|I|-1}{2} \rfloor + 4) \\
\notag
&& \times \sum_{ |K| \leq |I| -1}    \Big(    \frac{\eps  \cdot     |  \derm ( \Lie_{Z^K} A )  | }{(1+t+|q|)^{1-      c (\gamma)  \cdot c (\delta)  \cdot c( \lfloor \frac{|I|-1}{2} \rfloor) \cdot E ( \lfloor \frac{|I|-1}{2} \rfloor+ 4) \cdot \eps } \cdot (1+|q|)^{\gamma - 2\de }}    \\
&& + \frac{\eps \cdot  | \Lie_{Z^K}  A |  }{(1+t+|q|)^{1-      c (\gamma)  \cdot c (\delta)  \cdot c(\lfloor \frac{|I|-1}{2} \rfloor ) \cdot E (\lfloor \frac{|I|-1}{2} \rfloor + 4) \cdot \eps } \cdot (1+|q|)^{1+\gamma - 2\de }}   \Big) \;. 
\eea    
This was done in such a manner that we would have the correct control in the following estimate \eqref{estimateonthebadtermproductALtimesdermAusingbootstrapsoastoclosethegronwallinequalityonenergy} on $$\frac{ (1+t)}{\eps}  \cdot \Big(  \sum_{|K| \leq |I|}  | \Lie_{Z^K}    \big(  A_L   \cdot     \derm A    \big)    |   \Big)^{2} $$ with the \textit{right} factors so that the \textit{weighted} space integral $$\int_{\Sigma^{ext}_{\tau} }  \frac{ (1+t)}{\eps}  \cdot \Big(  \sum_{|K| \leq |I|}  | \Lie_{Z^K}    \big(  A_L   \cdot     \derm A    \big)    |   \Big)^{2}  \cdot w(q)  $$ would be suitably controlled in Lemma \ref{HardytpyeestimatesontehtermscontainingALinthebadstructureofsourcestermsforAusefultoobtainenergyestimates} after applying the Hardy type inequality (of Corollary \ref{HardytypeinequalityforintegralstartingatROm}). We have for $\ga \geq 3 \de $\,, and $0 < \de \leq \frac{1}{4}$\,,
                                         \bea\label{estimateonthebadtermproductALtimesdermAusingbootstrapsoastoclosethegronwallinequalityonenergy}
\notag
&&   \frac{ (1+t)}{\eps}  \cdot \Big(  \sum_{|K| \leq |I|}  | \Lie_{Z^K}    \big(  A_L   \cdot     \derm A    \big)    |   \Big)^{2} \\
\notag
&\les&  C(q_0)   \cdot  c (\delta) \cdot c (\gamma)  \cdot E ( 4)   \cdot \sum_{|K| = |I| }  \Big(  \frac{ \eps  \cdot     |  \derm ( \Lie_{Z^K} A )  |^2   }{ (1+t+|q|)  \cdot  (1+|q|)^{2\gamma - 4\de}  } \\
\notag
&&  +    \frac{\eps \cdot   |  \Lie_{Z^K} A_L |^2   }{(1+t+|q|)^{1-      c (\gamma)  \cdot c (\delta)   \cdot E ( 4) \cdot \eps } \cdot (1+|q|)^{2+2\gamma - 4\de }}       \Big) \\
\notag
&& +  C(q_0)   \cdot  c (\delta) \cdot c (\gamma) \cdot C( \lfloor \frac{|I|-1}{2} \rfloor) \cdot E ( \lfloor \frac{|I|-1}{2} \rfloor + 4) \\
\notag
&& \times \sum_{ |K| \leq |I| -1}    \Big(    \frac{\eps  \cdot     |  \derm ( \Lie_{Z^K} A )  |^2 }{(1+t+|q|)^{1-      c (\gamma)  \cdot c (\delta)  \cdot c( \lfloor \frac{|I|-1}{2} \rfloor) \cdot E ( \lfloor \frac{|I|-1}{2} \rfloor+ 4) \cdot \eps } \cdot (1+|q|)^{2\gamma - 4\de }}    \\
&& + \frac{\eps \cdot  | \Lie_{Z^K}  A |^2  }{(1+t+|q|)^{1-      c (\gamma)  \cdot c (\delta)  \cdot c(\lfloor \frac{|I|-1}{2} \rfloor ) \cdot E (\lfloor \frac{|I|-1}{2} \rfloor + 4) \cdot \eps } \cdot (1+|q|)^{2+2\gamma - 4\de }}   \Big) \;. 
\eea 
\end{proof}
In conclusion, both the leading term and the lower order terms have the right factors in the following estimate \eqref{HardyinequalityontheestimateonthebadtermproductALtimesdermAusingbootstrapsoastoclosethegronwallinequalityonenergy}, that would allow us to insert this successfully in the energy estimate \ref{Theveryfinal }. Indeed, for $\ga \geq 3\de$, and for $\eps$ small depending on $|I|$\;, on $\ga$ and on $\de$\;, we have
                                         \bea\label{HardyinequalityontheestimateonthebadtermproductALtimesdermAusingbootstrapsoastoclosethegronwallinequalityonenergy}
\notag
&&  \int_{\Sigma^{ext}_{\tau} } \frac{ (1+t)}{\eps}  \cdot \Big(  \sum_{|K| \leq |I|}  | \Lie_{Z^K}    \big(  A_L   \cdot     \derm A    \big)    |   \Big)^{2}  \cdot w(q)  \\
\notag
&\les&   \int_{\Sigma^{ext}_{\tau} } \Big[ C(q_0)   \cdot  c (\delta) \cdot c (\gamma)  \cdot C(|I|) \cdot  E ( \lfloor \frac{|I|}{2} \rfloor + 4)   \\
\notag
&& \times \sum_{|K| \leq |I| }  \Big(  \frac{ \eps  \cdot     |  \derm ( \Lie_{Z^K} A )  |^2  +  \eps \cdot |  \derm ( \Lie_{Z^K} h^1 )  |^2   }{ (1+t+|q|)   }   +    \frac{\eps \cdot   | \rderm ( \Lie_{Z^K} A ) |^2   }{(1+t+|q|)^{1-      c (\gamma)  \cdot c (\delta)   \cdot E ( 4) \cdot \eps } \cdot (1+|q|)}       \Big) \; \Big]  \cdot w(q)  \\
\notag
&& +   \int_{\Sigma^{ext}_{\tau} } \Big[  C(q_0)   \cdot  c (\delta) \cdot c (\gamma) \cdot C( |I|) \cdot E ( \lfloor \frac{|I|}{2} \rfloor + 4) \\
\notag
&& \times \sum_{ |K| \leq |I| -1}    \Big(    \frac{\eps  \cdot     |  \derm ( \Lie_{Z^K} A )  |^2 }{(1+t+|q|)^{1-      c (\gamma)  \cdot c (\delta)  \cdot c( \lfloor \frac{|I|-1}{2} \rfloor) \cdot E ( \lfloor \frac{|I|-1}{2} \rfloor+ 4) \cdot \eps } \cdot (1+|q|)^{2\gamma - 4\de }}      \Big)   \; \Big]   \cdot w(q)  \\
&&+   C(q_0)   \cdot  c (\delta) \cdot c (\gamma) \cdot C(|I|) \cdot E (  \lfloor \frac{|I|}{2} \rfloor  +4)  \cdot  \frac{\eps^3 }{(1+t)^{1-     c (\gamma)  \cdot c (\delta)  \cdot c(|I|) \cdot E ( \lfloor \frac{|I|}{2} \rfloor+ 4) \cdot \eps } }  \; .
\eea

\subsection{Upgrading the dispersive estimates for the Lie derivatives}\label{dealingwithLiederivativesupgrade}\

Unlike the case of the Einstein vacuum equations, and also unlike the case of the Einstein-Maxwell system, in the case of the Einstein-Yang-Mills equations, we need to have a \textit{more suitable} estimate for the following commutator term
\beaa
| \Lie_{Z^I}  ( g^{\la\mu} \derm_{\la}   \derm_{\mu}     A_{e_a} ) - g^{\la\mu}    \derm_{\la}   \derm_{\mu}  (  \Lie_{Z^I} A_{e_a}  ) | \; .
\eeaa
We notice that the term in the estimate of commutator that is behind the imposed Grönwall inequality, with an integral involving the gradients of the full components as previously used in the literature, is the term that appears with the weak factor $\frac{1}{|q|}$\;, and not the one that appears with the strong factor $\frac{1}{t}$\;. We realize that this troublesome term could be estimated in a more refined fashion as follows,
 \beaa
 \notag
 |\derm A_{\cal T} | &\les& \sum_{|I| \leq 1}  \frac{1}{(1+t+|q|)} \cdot | \Lie_{Z^I} A |  +  \sum_{ V^\prime \in \cal T } \sum_{|I| \leq 1}  \frac{1}{(1+|q|)} \cdot  | \Lie_{Z^I} A_{ V^\prime} | \; , \\
 \eeaa
 where now, we get an estimate where term with the weak factor is insensitive to the bad component $A_{{\underline{L}}}$ that bothered us, which when inserted to estimate the commutator term, leads to an estimate that is more refined (see the following estimate \eqref{commutationformaulamoreprecisetoconservegpodcomponentsstructure}). In fact, let  $\Phi_{\mu}$ be a one-tensor valued in the Lie algebra or a scalar. Then, we have for all $I$, and for any $V \in \cal T$, 
 \bea\label{commutationformaulamoreprecisetoconservegpodcomponentsstructure}
\notag
&&| \Lie_{Z^I}  ( g^{\la\mu} \derm_{\la}   \derm_{\mu}     \Phi_{V} ) - g^{\la\mu}    \derm_{\la}   \derm_{\mu}  (  \Lie_{Z^I} \Phi_{V}  ) |  \\
\notag
&\les&  \sum_{|K| < |I| }  | g^{\la\mu} \cdot \derm_{\la}   \derm_{\mu} (  \Lie_{Z^{K}}  \Phi_{V} ) | \\
\notag
&& +  \sum_{|J| + |K| \leq |I|, \; |K| < |I| }  \Big(   |   ( \Lie_{Z^{J}}   H)_{L  L} |   \cdot    \frac{1}{(1+t+|q|)} \cdot  \sum_{|M| \leq |K|+1}  | \derm ( \Lie_{Z^M}  \Phi )  |  \\
   \notag
 && +    |   ( \Lie_{Z^{J}}   H)_{L  L}   | \cdot   \frac{1}{(1+|q|)} \cdot   \sum_{|M| \leq |K|+1}    \sum_{ V^\prime \in \cal T } | \derm   ( \Lie_{Z^M}  \Phi _{V^\prime} ) |   \\
\notag
 && +   | ( \Lie_{Z^{J}}   H)_{L  \underline{L}} |   \cdot | \derm_{L}    \derm_{\underline{L}}  (  \Lie_{Z^{K}}  \Phi_{V} )  | \\
\notag
    &&  +    \frac{1}{(1 + t + |q|)  }  \cdot   |  ( \Lie_{Z^{J}}   H)_{L e_A }   | \cdot | \sum_{|M| \leq |K|+1} | \derm (  \Lie_{Z^{M}}  \Phi )| \\
    \notag
 && +       | m^{\mu \b}   ( \Lie_{Z^{J}}   H)_{\underline{L}\b}  \cdot    \derm_{L}   \derm_{\mu}     (  \Lie_{Z^{K}}  \Phi_{V} )  | + |   m^{\mu \b}   ( \Lie_{Z^{J}}   H)_{e_A \b}   \cdot   \derm_{e_A}   \derm_{\mu}    (  \Lie_{Z^{K}}  \Phi_{V} )   | \Big) \; .\\
  \eea
In the above estimate the terms with the weak factors do not see the bad component $A_{{\underline{L}}}$ (see the following estimate \eqref{Thecommutationformulawithpossibilityofsperationoftangentialcomponentsaswell}). Let  $\Phi_{\mu}$ be a tensor valued either in the Lie algebra or a scalar, satisfying the following tensorial wave equation
\beaa
 g^{\la\a} \derm_{\la}   \derm_{\a}   \Phi_{\mu}= S_{\mu} \, , 
\eeaa
where $S_{\mu}$ is the source term. Then, we have for any $V \in \cal T$,
 \bea\label{Thecommutationformulawithpossibilityofsperationoftangentialcomponentsaswell}
\notag
&&| \Lie_{Z^I}  ( g^{\la\mu} \derm_{\la}   \derm_{\mu}     \Phi_{V} ) - g^{\la\mu}    \derm_{\la}   \derm_{\mu}  (  \Lie_{Z^I} \Phi_{V}  ) |  \\
  \notag
   &\les&  \sum_{|K| < |I| }  | g^{\la\mu} \cdot \derm_{\la}   \derm_{\mu} (  \Lie_{Z^{K}}  \Phi_{V} ) | \\
   \notag
&&+  \frac{1}{(1+t+|q|)}  \cdot \sum_{|K|\leq |I|,}\,\, \sum_{|J|+(|K|-1)_+\le |I|} \,\,\, | \Lie_{Z^{J}} H |\, \cdot | \derm ( \Lie_{Z^K}  \Phi )  | \\
   \notag
&& +   \frac{1}{(1+|q|)}  \cdot \sum_{|K|\leq |I|,}\,\, \sum_{|J|+(|K|-1)_+\le |I|} \,\,\, | \Lie_{Z^{J}} H_{L  L} |\, \cdot  \big( \sum_{ V^\prime \in \cal T } | \derm ( \Lie_{Z^K}  \Phi _{V^\prime} )  | \:  \big) \; , \\
\eea
  where $(|K|-1)_+=|K|-1$ if $|K|\geq 1$ and $(|K|-1)_+=0$ if $|K|=0$. We now insert this estimate on the commutator term in the estimate of Lindblad-Rodnianski (see Corollary 7.2 of \cite{LR10}, that is at the heart of their argument for the case of the Einstein vacuum equations) that we can use in order to upgrade the estimate on the Lie derivatives (see the following estimate \eqref{estimatethatallowsupgradeincorporatingtermsfromthecommutationformula}). We have for $\gamma^\prime$ such that $-1 \leq \gamma^\prime < \gamma - \delta$, and for $ \delta <    1/2 $, and for all $U, V\in  \{L,\Lb,A,B\}$,
\bea\label{estimatethatallowsupgradeincorporatingtermsfromthecommutationformula}
   \notag
&& (1+t+|q|) \cdot |\varpi(q) \cdot \derm ( \Lie_{Z^J} A)_{V} (t,x)| \\
   \notag
 &\les&   c (\gamma^\prime) \cdot  c (\delta) \cdot c (\gamma) \cdot C ( |J|  ) \cdot E ( |J| + 4) \cdot \eps \, \\
\notag
&& + c (\gamma^\prime) \cdot c (\gamma)  \cdot c (\delta)  \cdot E ( 3)  \cdot \eps  \cdot  \int_0^t  \frac{1}{(1+\tau)} \cdot (1+\tau+|q|) \cdot \|\varpi(q)  \cdot  \derm  ( \Lie_{Z^J} A_{V}) (\tau,\cdot) \|_{L^\infty (\Sigma^{ext}_{\tau} )} d \tau \\
       \notag
    && + \sum_{|K| \leq |J|} \int_0^t (1+\tau) \cdot  \varpi(q) \cdot  \|    \Lie_{Z^K}  g^{\la\mu} \derm_{\la}   \derm_{\mu}  A_V (\tau,\cdot) \|_{L^\infty(\overline{D}_\tau)} d\tau \\
  \notag
& &+ \int_0^t   \frac {(1+\tau) \cdot  \varpi(q)}{(1+\tau+|q|)} \cdot    \sum_{|K|\leq |J|} \Big( \sum_{|J^{\prime}|+(|K|-1)_+\le |J|} \,\,\,
|\Lie_{Z^{J^{\prime}}} H|\cdot {|\derm ( \Lie_{Z^{K} } A) |}  \Big) d\tau \\
& &+  \int_0^t  \frac {(1+\tau) \cdot  \varpi(q)}{(1+|q|)} \cdot  \sum_{|K|\leq |J|} \Big( \sum_{|J^{\prime}|+(|K|-1)_+\leq |J|} \!\!\!\!\!| \Lie_{Z^{J^{\prime}} }H_{LL}|    \cdot \sum_{X \in \cal V} |\derm ( \Lie_{Z^{K}} A )_{X} | \Big) d\tau \; , 
\eea
where
\bea
\cal V :=  \begin{cases}  \cal T \;  ,\quad\text{if }\quad V \in \cal T \; ,\\
   \cal U \; , \,\quad\text{if }\quad V \in \cal U \; . \end{cases}   
\eea
As mentioned, the terms that have the good decay factor in $t$ do not generate a Grönwall type integral (see the following estimate \eqref{estimateonthegooddecayingpartofthecommutatorterm}). For $\gamma^\prime$ such that $-1 \leq \gamma^\prime < \gamma - \delta$, and $ \delta <    1/2 $, we have
  \bea\label{estimateonthegooddecayingpartofthecommutatorterm}
  \notag
   &&    \int_0^t   \frac {(1+\tau) \cdot  \varpi(q)}{(1+\tau+|q|)} \cdot    \sum_{|K|\leq |J|} \Big( \sum_{|J^{\prime}|+(|K|-1)_+\le |J|} \,\,\,
|\Lie_{Z^{J^{\prime}}} H|\cdot {|\derm ( \Lie_{Z^{K} } A) |}  \Big) d\tau \\
&\les&  c (\delta) \cdot c (\gamma) \cdot C ( |J| ) \cdot E ( |J| + 2)  \cdot \eps^2 \; .
\eea
However, the terms that have the weak decaying factor in $|q|$ generate a Grönwall integral (see the following estimate \eqref{estimateonthebadpartofthecommutatortermtobtainaGronwallforcomponents}), but this time, it does not involve the full components, but only the good ones $A_{\cal T}$\;, for which the non-linear wave equations do not have in their sources the troublesome term $A_{e_a}  \cdot     \derm A_{e_a} $\;. For $M \leq \eps$, under the induction hypothesis on both $A$ and on $h^1$, for  $|K|\leq |J| -1$, we have for $\gamma^\prime$ such that $-1 \leq \gamma^\prime < \gamma - \delta$, and $ \delta \leq  \frac{1}{4} $, and for $\cal V \in \{\cal T\; , \cal U\}$,
\bea\label{estimateonthebadpartofthecommutatortermtobtainaGronwallforcomponents}
\notag
& &  \int_0^t  \frac {(1+\tau) \cdot \varpi(q)}{1+|q|} \cdot   \sum_{|K|\leq |J|} \Big( \sum_{|J^{\prime}|+(|K|-1)_+\leq |J|} \!\!\!\!\!| \Lie_{Z^{J^{\prime}} }H_{LL}|   \cdot\sum_{X \in \cal V} |\derm ( \Lie_{Z^{K}} A )_{X} | d\tau  \Big)\; . \\
\notag
&\leq &   \int_0^t  C(q_0)   \cdot  c (\delta) \cdot c (\gamma) \cdot C(|J|) \cdot E ( |J| + 3)  \cdot \frac{\eps \cdot  \varpi(q) }{(1+t+|q|)^{1-      c (\gamma)  \cdot c (\delta)  \cdot c(|J|) \cdot E ( |J|+ 3) \cdot \eps } \cdot (1+|q|)^{2+\gamma - 2 \de }}  d\tau \\
&&+    \int_0^t    c (\delta) \cdot c (\gamma) \cdot E (  4 )  \cdot \frac{\eps }{(1+t+|q|) }     \cdot   (1+t+|q|) \cdot \varpi(q) \cdot \sum_{|K| = |J|} \sum_{X \in \cal V} |\derm ( \Lie_{Z^{K}} A )_{X} |  d\tau \; .
 \eea
This allows us to establish a Grönwall type inequality that enables us to upgrade first for the good components separately, and consequently for $\derm A_{e_a} $ (see the following estimate \ref{GronwallinequalitongradientofAandh1withliederivativesofsourceterms}). Let $ 0\leq \delta \leq  \frac{1}{4} $, and $\ga > \de$, and $M \leq \eps \leq 1$. We assume the induction hypothesis holding true for both $A$ and on $h^1$, for all $|K|\leq |J| -1$. Let
\beaa
\varpi(q) :=\begin{cases}
(1+|q|)^{1+\gamma^\prime},\quad\text{when }\quad q>0\; , \\
     1 \,\quad\text{when }\quad   q<0 \;, \end{cases} 
\eeaa
and let
 \beaa
\cal V :=  \begin{cases}  \cal T \;  ,\quad\text{if }\quad V \in \cal T \; ,\\
   \cal U \; , \,\quad\text{if }\quad V \in \cal U \; . \end{cases}   
\eeaa
We have for $\ga^\prime =  \gamma - 2 \de$, the following estimate in the exterior for $A$, 
            \bea\label{GronwallinequalitongradientofAandh1withliederivativesofsourceterms}
   \notag
&& (1+t+|q|) \cdot |\varpi(q) \cdot \derm ( \Lie_{Z^J} A)_V | \\
   \notag
   &\les&   C(q_0)  \cdot  c (\delta) \cdot c (\gamma) \cdot C ( |J|  ) \cdot E ( |J| + 4) \cdot \eps \cdot (1+t)^{     c (\gamma)  \cdot c (\delta)  \cdot c(|J|) \cdot E ( |J|+ 4) \cdot \eps }  \\
\notag
&& +  c (\delta) \cdot c (\gamma) \cdot E (  4 )  \cdot \eps  \cdot  \int_0^t  \frac{1}{(1+\tau)} \cdot (1+\tau+|q|) \cdot  \sum_{X \in \cal V}  \|\varpi(q)  \cdot  \derm  ( \Lie_{Z^J} A)_{X} (\tau,\cdot) \|_{L^\infty (\Sigma^{ext}_{\tau} )} d \tau \\
    && + \sum_{|K| \leq |J|} \int_0^t (1+\tau) \cdot  \varpi(q) \cdot  \|    \Lie_{Z^K}  g^{\la\mu} \derm_{\la}   \derm_{\mu}  A_V (\tau,\cdot) \|_{L^\infty(\overline{D}_\tau)} d\tau \; .
 \eea
However, in order for us to apply the upgrade on $\derm A_{e_a} $ in an upgrade for $\derm A_{{\underline{L}}}$\;, we estimate the product $$  \sum_{|K| + |I| \leq |J| }   | \Lie_{Z^K}  A_{e_a}   |  \cdot | \derm   ( \Lie_{Z^I} A_{e_a} )  | \; ,$$ which involves estimating $ | \Lie_{Z^K}  A_{e_a}   | $ (which is not the covariant gradient). So how to translate an estimate on the covariant gradient of specific components $|\derm ( \Lie_{Z^I}  A_{\cal T} )|$\;, into an estimate on the partial derivative $| \pa_r ( \Lie_{Z^I}  A_{e_a} ) |$\;, which is the term that is actually needed to be estimated in order to integrate to estimate $\Lie_{Z^I}  A_{e_a} $? We use the specific fact that it is not just any component, but an  $A_{e_a}$ component, and we notice that $\derm_{r} e_a = 0 $\;. This allows us to get the right estimate on this product (see the following estimate \eqref{upgradeddecayestimateonthetermAe_agradAe_a}), which we shall use to upgrade for the full components (see the following estimate \eqref{SteptwoforinductonforA}). Let $ 0\leq \delta \leq  \frac{1}{4} $, and $\ga > \de$, and $M \leq \eps \leq 1$. We assume the induction hypothesis holding true for both $A$ and on $h^1$, for all $|K|\leq |J| -1$. Then, in the exterior region $\overline{C}$, we have the following estimate
     \bea\label{upgradeddecayestimateonthetermAe_agradAe_a}
     \notag
  &&   \sum_{|K| + |I| \leq |J| }   | \Lie_{Z^K}  A_{e_a}   |  \cdot | \derm   ( \Lie_{Z^I} A_{e_a} )  |    \\
  \notag
       &\leq&  C(q_0)  \cdot  c (\delta) \cdot c (\gamma) \cdot C ( |J|  ) \cdot E ( |J| + 4) \cdot     \frac{\eps}{ (1+t)^{  2-   c (\gamma)  \cdot c (\delta)  \cdot c(|J|) \cdot E ( |J|+ 4) \cdot \eps }\cdot  (1+|q|)^{1+2\ga - 4 \de}  }  \; .\\
\eea
Let $ 0 < \delta \leq  \frac{1}{4} $, and $\ga \geq 3 \de$, and $M \leq \eps \leq 1$. We assume the induction hypothesis holding true for both $A$ and on $h^1$, for all $|K|\leq |J| -1$. Then, in the exterior region $\overline{C}$, we have the following estimate on all the components of the Einstein-Yang-Mills potential,
      \bea\label{SteptwoforinductonforA}
   \notag
&& | \derm ( \Lie_{Z^J} A) | \\
   \notag
        &\les&   C(q_0)  \cdot  c (\delta) \cdot c (\gamma) \cdot C ( |J|  ) \cdot E ( |J| + 4) \cdot \eps \cdot \frac{1}{ (1+t)^{  1-   c (\gamma)  \cdot c (\delta)  \cdot c(|J|) \cdot E ( |J|+ 4) \cdot \eps }\cdot  (1+|q|)^{1+\ga - 2 \de}  }  \; .\\
 \eea
We point out that here, the term $ \sum_{|K| + |I| \leq |J| }  | \Lie_{Z^K} A_L  | \cdot    | \derm  ( \Lie_{Z^I}  A  ) | $ is not a problem this time, since we can estimate the product successfully using the a priori estimates (see the following estimate \eqref{upgradeddecayestimateonthetermALgradientA}) and this is thanks to the Lorenz gauge estimate on $A_L$ (see estimate \eqref{estimategoodcomponentspotentialandmetric}). Let $ 0\leq \delta \leq  \frac{1}{4} $, and $\ga > \de$, and $M \leq \eps \leq 1$. We assume the induction hypothesis holding true for both $A$ and on $h^1$, for all $|K|\leq |J| -1$. 
   
  Then, in the exterior region $\overline{C}$, we have the following estimate,
        \bea\label{upgradeddecayestimateonthetermALgradientA}
        \notag
 && \sum_{|K| + |I| \leq |J| }  | \Lie_{Z^K} A_L  | \cdot    | \derm  ( \Lie_{Z^I}  A  ) |  \\
 \notag
    &\les & C(q_0)   \cdot  c (\delta) \cdot c (\gamma) \cdot C(|J|) \cdot E ( |J| + 3)  \cdot \frac{\eps }{(1+t+|q|)^{2- c (\gamma)  \cdot c (\delta)  \cdot c(|J|) \cdot E ( |J|+ 3) \cdot \eps } \cdot (1+|q|)^{1+2\gamma -4\de }} \; .\\
  \eea

\subsection{The closure of the bootstrap argument on the energy}\

\begin{proposition} \label{THEpropositiontoclosetheargumenttoboundtheenergy}
Let $N \geq 11$\;. We have for $\ga \geq 3 \de $\,, for $0 < \de \leq \frac{1}{4}$\,, for $\eps$ small, enough depending on $q_0$\,, on $\ga$\;, on $\de$\;, on $N$ and on $\mu < 0$\;, and for $M \leq \eps^2 \leq 1$\;, and for $
\overline{ \E}_{N+2} \leq \eps$ (defined in \eqref{definitionoftheenergynormforinitialdata}), that under the bootstrap assumption \eqref{aprioriestimate}, we have
  \beaa
   \notag
     \E_{N} (t) &\leq&  \frac{E(N) }{2} \cdot \eps  \cdot  (1+t)^{ \de }  \; .
\eeaa
\end{proposition}

\begin{proof}
We have for $\ga \geq 3 \de $\,, for $0 < \de \leq \frac{1}{4}$\,, and for $\eps$ small, enough depending on $q_0$\,, on $\ga$\;, on $\de$\;, on $|I|$ and on $\mu$\;, and for $M \leq \eps^2 \leq 1$\;, that 
     \bea\label{Theinequalityontheenergytobeusedtoapplygronwallrecursively}
   \notag
 &&     \E_{|I|} (t_2) \\
   \notag
&\les&    C(q_0)   \cdot  c (\delta) \cdot c (\gamma) \cdot C(|I|) \cdot \big(    E (|I|  ) +E (  \lfloor \frac{|I|}{2} \rfloor  +6)   \big)     \cdot  \Big[  \sum_{|K| \leq |I |}      \int_{t_1}^{t_2}        \frac{\eps  }{(1+t)}   \cdot \E_{|K|} (t) \cdot   dt   \\  
&& +     \sum_{|K| \leq |I | -1}      \int_{t_1}^{t_2}         \frac{\eps  }{(1+t)^{1-      c (\gamma)  \cdot c (\delta)  \cdot c(|I| ) \cdot E ( \lfloor \frac{|I|}{2} \rfloor  +4) \cdot \eps } }   \cdot \E_{|K|} (t) \cdot   dt  \\
&&+   \eps^2  \cdot (1+t_2)^{ c (\gamma)  \cdot c (\delta)  \cdot c(|I|) \cdot E ( \lfloor \frac{|I|}{2} \rfloor+ 5) \cdot \eps }   +   \eps \cdot \E_{|I|+2} (t_1)  \Big]   \; .
\eea
This leads to
  \bea\label{Theboundontheenergybytimetbyepsilonandboostrapassumptionsandinitialdata}
   \notag
     \E_{|I|} (t_2) &\les&    (1+t_2)^{ C(q_0)   \cdot  c (\delta) \cdot c (\gamma) \cdot C( |I|   )  \cdot \big(    E (|I| ) +E (  \lfloor \frac{|I| }{2} \rfloor  +6)   \big)      \cdot \eps } \\
        \notag
     && \times C(q_0)   \cdot  c (\delta) \cdot c (\gamma) \cdot C( |I|   )  \cdot   \big(    E (|I| ) +E (  \lfloor \frac{|I| }{2} \rfloor  +6)   \big)  \cdot   \Big[  \eps     +  \E_{|I|+2} (t_1)  \Big]  \; .
\eea
Hence, we get the result.
\end{proof}
Finally, we proved our Theorem \ref{Thetheoremofexteriorstabilityfornequalthree} through a continuity argument by upgrading the energy estimate in \eqref{aprioriestimate}.

\end{document}